\documentclass[11pt]{amsart}
\usepackage{amssymb,amsmath,amscd,latexsym,epsfig,color,hyperref,graphics}
\usepackage[all,cmtip]{xy}

\newtheorem{theorem}{Theorem}[section]
\newtheorem{proposition}[theorem]{Proposition}
\newtheorem{corollary}[theorem]{Corollary}

\newtheorem{lemma}[theorem]{Lemma}

\theoremstyle{definition}

\newtheorem{remark}[theorem]{Remark}
\newtheorem{example}[theorem]{Example}

\def\NN{\ensuremath{\mathbb{N}}}
\def\ZZ{\ensuremath{\mathbb{Z}}}
 
\def\RR{\ensuremath{\mathbb{R}}}

\newcommand{\PP}{{\mathbb P}}

\newcommand{\B}{{\mathcal B}}

\def\A{\ensuremath{\mathcal{A}}}

\def\xx{\ensuremath{{\bf{x}}}}

\def\proof{\smallskip\noindent {\it Proof: \ }}
\def\endproof{\hfill$\square$\medskip}

\newcommand{\comment}[1]{}

\newcommand{\bmat}[1]{\begin{bmatrix}#1\end{bmatrix}}

\newcommand{\XX}{\textbf{x}}
\DeclareMathOperator{\PGL}{PGL}

\DeclareMathOperator{\tin}{in}

\title{A Hilbert Scheme in Computer Vision}
\author{Chris Aholt, Bernd Sturmfels and Rekha Thomas}
\address{Chris Aholt, Mathematics, University of Washington, Seattle, WA 98195}
\email{aholtc@uw.edu}
\address{Bernd Sturmfels, Mathematics, Univ.~of California,
Berkeley, CA 94720}
\email{bernd@math.berkeley.edu}
\address{Rekha Thomas, Mathematics, University of Washington, Seattle, WA 98195}
\email{rrthomas@uw.edu}

\begin{document}

\begin{abstract}
Multiview geometry is the study of 
two-dimensional images of three-dimensional scenes, a foundational subject in computer vision.
We determine a universal Gr\"obner basis for the multiview ideal of $n$ generic cameras.
As the cameras move, the multiview varieties vary in a family of dimension $11n-15$.
This family is the distinguished component of a multigraded Hilbert scheme
with a unique Borel-fixed point.
We present a combinatorial study
of ideals lying on that Hilbert scheme.

\end{abstract}
\maketitle

\section{Introduction}

Computer vision is based on mathematical foundations known as
{\em multiview geometry} \cite{FL, Grosshans} or {\em epipolar geometry} \cite[\S 9]{HartleyZisserman}.
In that subject one studies the
space of pictures of three-dimensional objects seen from $n \geq 2$ cameras.
Each camera is represented by a $3 \times 4$-matrix $A_i$ of rank $3$. The matrix
specifies a linear projection from  $\PP^3$ to $\PP^2$, which is well-defined on
$\PP^3 \backslash \{f_i\}$, where the  focal point $f_i $ is represented by a
generator of  the kernel of~$A_i$.
   
   The space of pictures from the $n$ cameras is the image of the rational~map 
 \begin{equation}
 \label{eq:phiA}
\phi_A \,:\,\,\PP^3 \,\dashrightarrow\, (\PP^2)^n, \,\,\,\, \XX \,\mapsto \, (A_1 \XX, A_2 \XX, \ldots,A_n \XX).
\end{equation}
The closure of this image is an algebraic variety, denoted $V_A$
and called the {\em multiview variety} of the given $n$-tuple of $3 \times 4$-matrices $A = (A_1,A_2,\ldots,A_n)$.
In geometric language, the multiview variety $V_A$ is the blow-up of
$\PP^3$ at the cameras $f_1,\ldots,f_n$, and we here study 
 this threefold as a subvariety of  $(\PP^2)^n$.

The {\em multiview ideal} $J_A$ is the prime ideal of all polynomials that vanish on the
multiview variety $V_A$. It lives in a polynomial ring $K[x,y,z]$ in $3n$ unknowns
$(x_i,y_i,z_i)$, $ i = 1,2,\ldots,n $, that serve as coordinates on $(\PP^2)^n$.
In Section 2 we give a determinantal representation of $J_A$ for generic $A$,
and identify a universal Gr\"obner basis consisting of
multilinear polynomials of degree $2$, $3$ and $4$.
This extends previous results of
Heyden and {\AA}str{\"o}m \cite{HA}.

The multiview ideal $J_A$ has 
a distinguished initial monomial ideal $M_n$ that is independent
of $A$, provided the  configuration $A$ is generic.
Section 3 gives an explicit description of $M_n$ and shows
that it is the unique Borel-fixed ideal with its $\ZZ^n$-graded Hilbert function.
Following \cite{CS}, we introduce the multigraded
Hilbert scheme $\mathcal{H}_n$
which parametrizes $\ZZ^n$-homogeneous ideals in
$K[x,y,z]$ with the same Hilbert function as $M_n$.
We show in Section 6 that, for $n \geq 3$, $\mathcal{H}_n$ has a distinguished component
of dimension $11n-15$ which compactifies the space
of camera positions studied in computer vision.
For two cameras, that space
 is an irreducible cubic hypersurface in $ \mathcal{H}_2 \simeq  \PP^8$.

Section 4 concerns the case when $n \leq 4$ and
the focal points $f_i$ are among the coordinate points 
$(1{:}0{:}0{:}0), \ldots, (0{:}0{:}0{:}1)$. Here the multiview variety $V_A$
is a toric threefold, and its degenerations are parametrized by a certain
toric Hilbert scheme inside $\mathcal{H}_n$. Each initial monomial
ideal of the toric ideal $J_A$ corresponds to a three-dimensional mixed subdivision 
as seen in Figure~\ref{V3_J8_Blowup}.
A classification of such mixed subdivisions for $n=4$ is given in
Theorem~\ref{thm:1068}.

\begin{figure}
\includegraphics[width=0.44\linewidth]{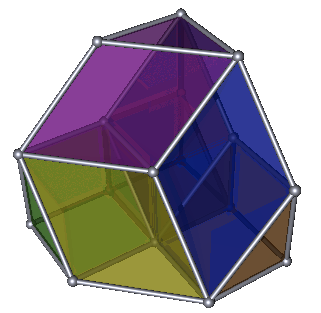} 
\includegraphics[width=0.55\linewidth]{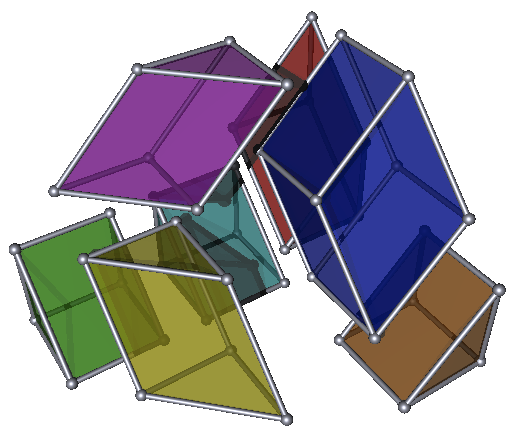}
\caption{A multiview variety $V_A$ for $n = 3$ cameras degenerates into
six copies of $\PP^1 {\times} \PP^2$ and one copy of $\PP^1 {\times} \PP^1 {\times} \PP^1$.}
\label{V3_J8_Blowup}
\end{figure}

In Section 5 we place our $n$ cameras on a line in $\PP^3$.
Moving them very close to each other on that line induces
a two-step degeneration of the form
\begin{equation}
\label{eq:TriBiMono}
 \hbox{trinomial ideal}  
\ \longrightarrow \
 \hbox{binomial ideal}
\  \longrightarrow \
 \hbox{monomial ideal}.
 \end{equation}
We present an in-depth combinatorial study of this curve of multiview ideals.

In Section 6 we finally define the Hilbert scheme $\mathcal{H}_n$,
and we construct the space of camera positions 
as a GIT quotient of a Grassmannian.
Our main result (Theorem \ref{thm:component})
states that the latter is an irreducible
component of~$\mathcal{H}_n$.
As a key step in the proof,
the tangent space of $\mathcal{H}_n$
at the monomial ideal in (\ref{eq:TriBiMono}) is computed 
and shown to have the correct dimension $11n-15$.
Thus, the curve (\ref{eq:TriBiMono})
consists of smooth points on the distinguished component of~$\mathcal{H}_n$.
For $n \geq 3$, our Hilbert scheme has multiple components.
This is seen from our classification of 
monomial ideals on $\mathcal{H}_3$, which relates closely to
\cite[\S 5]{CS}.

\bigskip

\noindent{\bf Acknowledgments}. Aholt and Thomas thank Fredrik Kahl for hosting them
at  Lund in February 2011 and pointing them to the work of Heyden and {\AA}str{\"o}m.
  They also thank Sameer Agarwal for introducing them to problems in computer vision 
  and continuing to advise them in this field.
Sturmfels thanks the Mittag-Leffler Institute, where this project started, 
and MATHEON Berlin for their hospitality.
All three authors were partially supported by the US National Science Foundation.
We are indebted to the makers of the software packages 
{\tt CaTS}, {\tt Gfan}, {\tt Macaulay2} and {\tt Sage}
which allowed explicit computations that were crucial in discovering our results.

\section{A universal Gr\"obner basis}

Let $K$ be any algebraically closed field, $n \geq 2$, and consider the map $\phi_A$
defined as in (\ref{eq:phiA})  by a tuple $A = (A_1,A_2,\ldots,A_n) $ of $3 \times 4$-matrices 
of rank $3$ with entries in $K$.
   The subvariety $V_A = \overline{{\rm image}(\phi_A)}$ of $(\PP^2)^n$
is the {\em multiview variety}, and its ideal $J_A \subset K[x,y,z]$
is the {\em multiview ideal}. Note that $J_A$ is prime because
its variety $V_A$ is the image
under $\phi_A$ of an irreducible variety.

We say that the camera configuration $A$ is {\em generic} if all $4 \times 4$-minors of the
$(4 \times 3n)$-matrix $ \bmat{ A_1^T \! & \! A_2^T \! & \!  \cdots \! & \! A_n^T }$
   are non-zero. In particular, if $A$ is generic then the focal points of the $n$ cameras are 
   pairwise distinct in $\PP^3$.
For any subset $\sigma = \{\sigma_1,\ldots,\sigma_s\} \subseteq [n]$ 	
we consider the $3s \times (s+4)$-matrix
$$ A_\sigma \,\,\, :=\,\,\, \begin{bmatrix}
A_{\sigma_1} & p_{\sigma_1} & \mathbf{0} & \cdots & \mathbf{0}\\
A_{\sigma_2} & \mathbf{0} & p_{\sigma_2} & \ddots & \mathbf{0} \\
\vdots& \vdots & \ddots &\ddots & \vdots \\
A_{\sigma_s} & \mathbf{0} & \cdots & \mathbf{0} &  p_{\sigma_s}
\end{bmatrix},$$
where $p_i:= \bmat{x_i \! &  \! y_i \! & \! z_i}^T$ for $i \in [n]$.
Assuming $s \geq 2$, each maximal minor of $A_\sigma$ is
a homogeneous polynomial of degree $s= |\sigma|$ that is linear
in $p_i$ for $i \in \sigma$.  Thus for $ s= 2,3,\ldots$
these polynomials are bilinear, trilinear, etc.
The matrix $A_\sigma$ and its maximal minors are considered frequently
in multiview geometry \cite{HartleyZisserman, HA}.
Recall that a {\em universal Gr\"obner basis} of an ideal is a subset that is a Gr\"obner basis of the ideal under all term orders. 
The following  is the main result in this section.

\begin{theorem}
\label{thm:UGB}
If $A$ is generic then
the maximal minors of the matrices $A_\sigma$ for $2 \leq |\sigma| \leq 4$ 
form a universal Gr\"obner basis of the multiview ideal $J_A$.
\end{theorem}

The proof rests on a sequence of lemmas.
Here is the most basic~one.

\begin{lemma}
\label{lem:easyinclusion}
The maximal minors of $A_\sigma$ for $|\sigma| {\geq} 2$ lie in the prime ideal~$J_A$.
\end{lemma}

\begin{proof} If $ (p_1,\ldots,  p_n) \in (K^3)^n$ 
represents a point in ${\rm image}(\phi_A)$
then there exists a non-zero vector $q \in K^4$ 
and non-zero scalars $c_1,\ldots,c_n \in K$ such that
$A_i q = c_i  p_i$ for $i = 1,2,\ldots,n$.
This means that the columns of $A_\sigma$ are linearly dependent.
Since $A_\sigma$ has at least as many rows as columns, 
the maximal minors of $A_\sigma$ must vanish at every point $ p \in V_A$.
\end{proof}

Later we shall see that when $A$ is generic, $J_A$ has only one initial monomial ideal up to symmetry.
We now identify that ideal.
Let $M_n$ denote the  ideal in $K[x,y,z]$  generated by the ${n\choose 2}$ quadrics $x_ix_j$, the 
$3{n\choose 3}$ cubics $x_iy_jy_k$, 
and the ${n\choose 4}$ quartics $y_iy_jy_ky_l$, where $i,j,k,l$ runs over distinct indices in $[n]$.

We fix the lexicographic term order $\prec$ on $K[x,y,z]$ which is specified by
$x_1 {\succ} \cdots {\succ} x_n {\succ} 
y_1 {\succ} \cdots {\succ} y_n {\succ} z_1 {\succ} \cdots {\succ} z_n$.
Our goal is to prove that the initial monomial ideal $\tin_\prec(J_A)$ is equal to $M_n$.
We begin with the easier inclusion.

\newpage

\begin{lemma} \label{lem:generic1}
If $A$ is generic then $M_n \subseteq \tin_\prec(J_A) $.
\end{lemma}

\begin{proof}
The generators of $M_n$ are the quadrics $x_ix_j$, the cubics $x_iy_jy_k$, and the quartics $y_iy_jy_ky_l$.
By Lemma \ref{lem:easyinclusion}, it suffices to show that these are the initial monomials
of maximal minors of $A_{\{ij\}}$, $A_{\{ijk\}}$ and $A_{\{ijkl\}}$ respectively.

For the quadrics this is easy. The matrix $A_{\{ij\}}$ is square and we have
\begin{equation}
\label{eq:zwei}
 {\rm det}(A_{\{ij\}}) 
\,\, = \,\, {\rm det} \begin{bmatrix}
A_i^1 & \! x_i & \! 0\\
A_i^2 & \! y_i & \! 0\\
A_i^3 & \! z_i & \! 0\\
A_j^1 & \! 0 & \! x_j\\
A_j^2 & \! 0 & \! y_j\\
A_j^3 & \! 0 & \! z_j
\end{bmatrix}  \,= \,\, 
 {\rm det} \! \begin{bmatrix} A_i^2 \\ A_i^3 \\ A_j^2 \\ A_j^3 \end{bmatrix}  \! x_i x_j \,+\, \hbox{lex.~lower terms}. 
\end{equation}
where $A_t^r$ is the $r$th row of $A_t$.
The coefficient
of $x_i x_j$ is non-zero because $A$ was assumed to be generic.
For the cubics, we consider the $9\times 7$-matrix
\begin{equation}
\label{eq:drei}
A_{\{ijk\}} \quad = \quad  \begin{bmatrix}
A_i & p_i & 0 & 0\\
A_j & 0 & p_j & 0\\
A_k & 0 & 0 & p_k
\end{bmatrix}.
\end{equation}
Now, $x_i y_j y_k$ is the lexicographic initial monomial of the
$7\times 7$-determinant formed by removing the fourth and seventh rows of $A_{\{ijk\}}$.
Here we are using that, by genericity, the vectors $A_i^2, A_i^3, A_j^3, A_k^3$
are linearly independent.

Finally, for the quartic monomial $y_iy_jy_ky_l$ we consider the $12\times 8$ matrix
\begin{equation}
\label{eq:vier}
A_{\{ijkl\}} \quad =  \quad \begin{bmatrix}
A_i & p_i & 0 & 0 & 0\\
A_j & 0 & p_j & 0 & 0\\
A_k & 0 & 0 & p_k & 0\\
A_l & 0 & 0 & 0 & p_l
\end{bmatrix}.
\end{equation}
Removing the first row from each of the four blocks, we obtain an $8 \times 8$-matrix
whose determinant has $y_iy_jy_ky_l$ as its lex.~initial monomial.
\end{proof}

The next step towards our proof  of Theorem \ref{thm:UGB} is to express
the multiview variety $V_A$ as a projection of a 
diagonal embedding of $\PP^3$. This will put us in a position to
utilize the results of Cartwright and Sturmfels in \cite{CS}.

We extend each camera matrix $A_i$ to an invertible
$4 \times 4$-matrix $B_i = \bmat{b_i\\A_i} $ by adding a row $b_i$ at the top.
Our diagonal embedding of $\PP^3$ is the map
\begin{equation}
 \label{eq:psiB}
 \psi_B: \PP^3 \,\to\, (\PP^3)^n, \,\,\,\, \XX \,\mapsto\, (B_1 \XX, B_2 \XX, \ldots,B_n \XX).
 \end{equation}

Let $V^B := {\rm image}(\psi_B) \subset (\PP^3)^n$ and $J^B \subset K[w,x,y,z]$ its prime ideal. 
Here $(w_i:x_i:y_i:z_i)$ are coordinates on the $i$th copy of $\PP^3$ and $(w,x,y,z)$ are
coordinates on $(\PP^3)^n$. The ideal $J^B$ is generated by the $2 \times 2$-minors of
\begin{equation} \label{eq:inverseB}
 \left[
 B_1^{-1} \! \begin{bmatrix} w_1 \\ x_1 \\ y_1 \\ z_1 \end{bmatrix} \,\,
B_2^{-1} \! \begin{bmatrix} w_2 \\ x_2 \\ y_2 \\ z_2 \end{bmatrix} \, \cdots \,\,\,
B_n^{-1} \! \begin{bmatrix} w_n \\ x_n \\ y_n \\ z_n \end{bmatrix} \,\right] .
\end{equation}
This is a  $4 \times n$-matrix.
Now consider the coordinate projection
$$ \pi \,:\, (\PP^3)^n \dashrightarrow (\PP^2)^n \, , \,\,\,
(w_i:x_i:y_i:z_i) \mapsto (x_i:y_i:z_i) \textup{ for }  i=1,\ldots, n. $$
The composition $\pi \circ \psi_B$ is a rational map, and it coincides with $\phi_A$ 
on its domain of definition $\PP^3 \backslash \{f_1, \ldots, f_n \}$.
Therefore, $V_A = \overline{\pi(V^B)}$ and 
\begin{equation}
\label{eq:elimideal}
 J_A \,\, =  \,\, J^B \cap K[x,y,z] . 
 \end{equation}

The polynomial ring $K[w,x,y,z]$ admits the natural $\mathbb Z^n$-grading $\deg(w_i) = \deg(x_i) = \deg(y_i) = \deg(z_i) = e_i$ where $e_i$ is the standard unit vector in $\RR^n$. Under this grading, $K[w,x,y,z]/J^B$ has the multigraded Hilbert function
$$ \mathbb N^n \to \mathbb N, \,\,\,(u_1, \ldots, u_n) \mapsto \left( \begin{array}{c} 
u_1+\cdots+u_n+3 \\ 3 \end{array} \right).$$
The multigraded Hilbert scheme $H_{4,n}$ which parametrizes $\ZZ^n$-homogeneous ideals 
in $K[w,x,y,z]$ with that Hilbert function was studied in \cite{CS}. 
More generally, the multigraded Hilbert scheme $H_{d,n}$ represents
degenerations of the diagonal $\PP^{d-1}$ in $(\PP^{d-1})^n$ for any $d$ and $n$.
For the general definition  of
multigraded Hilbert schemes see \cite{HaimanSturmfels}.
  It was shown in \cite{CS} that $H_{d,n}$ has a unique Borel-fixed ideal $Z_{d,n}$.
  Here {\em Borel-fixed} means that $Z_{d,n}$ is stable
    under the action of $\B^n$ where $\B$ is the group of
 lower triangular matrices in ${\PGL}(d,K)$.
   Here is what we shall need
  about the monomial ideal $Z_{4,n}$.
  
  \begin{lemma} \label{lem:generators of Z}  {\rm (Cartwright-Sturmfels \cite[\S 2]{CS} and 
  Conca \cite[\S 5]{Conca})}
\begin{enumerate}
\item The unique Borel-fixed monomial ideal $Z_{4,n}$ on $H_{4,n}$ is generated by the following 
monomials where $i,j,k,l$ are distinct indices in $[n]$:
\begin{center}
$\begin{array}{l}
w_iw_j, \,w_ix_j,\, w_iy_j, \,x_ix_j,\,\,
\,x_iy_jy_k, \,\,
y_iy_jy_ky_l.
\end{array}$
\end{center}
\item This ideal $Z_{4,n}$ is the lexicographic 
 initial ideal of $J^B$ when $B$ is sufficiently generic. The lexicographic order here is $w \succ x \succ y \succ z$ with each block ordered lexicographically in increasing order of index.
\end{enumerate}
\end{lemma}

Using these results, it was deduced in \cite{CS} that all ideals on $H_{4,n}$ are radical and Cohen-Macaulay, and that $H_{4,n}$ is connected. We now use this distinguished Borel-fixed ideal $Z_{4,n}$ to 
prove the equality in Lemma \ref{lem:generic1}.

\begin{lemma} \label{lem:generic2}
If $A$ is generic then $M_n = \tin_\prec(J_A) $.
\end{lemma}

\begin{proof}
We fix the lexicographic term order $\prec$ on $K[w,x,y,z]$
and its restriction to $K[x,y,z]$.  Lemma \ref{lem:generators of Z}  (1)
shows that $\,M_n = Z_{4,n} \cap K[x,y,z] $.
 Lemma \ref{lem:generators of Z}  (2) states that $Z_{4,n} = \tin_{\prec}(J^B)$ when $B$ is generic.
   The lexicographic order has the important property that it allows the operations of taking initial ideals and intersections to commute \cite[Chapter 3]{CLO}. Therefore,
\begin{align*}
\tin_{\prec}(J_A) &\,=\, \tin_{\prec}(J^B\cap K[x,y,z]) & \\
&\,=\, \tin_{\prec}(J^B)\cap K[x,y,z] \\
&\,=\, Z_{4,n} \cap K[x,y,z] \,\,= \,\, M_n.
\end{align*}
This identity is valid whenever the  conclusion of
Lemma \ref{lem:generators of Z}  (2) is true.
We claim that, for this to hold, the appropriate genericity notion for $B$  is 
 that all $4 \times 4$-minors of the
$(4 \times 4n)$-matrix $ \bmat{ B_1^T \! & \! B_2^T \! & \!  \cdots \! & \! B_n^T }$
   are non-zero. Indeed, under this hypothesis, the maximal minors
   of the $4s \times (s+4)$-matrix
$$ \quad B_\sigma \,\,\, :=\,\,\, \begin{bmatrix}
B_{\sigma_1} & \tilde p_{\sigma_1} & \mathbf{0} & \cdots & \mathbf{0}\\
B_{\sigma_2} & \mathbf{0} & \tilde p_{\sigma_2} & \ddots & \mathbf{0} \\
\vdots& \vdots & \ddots &\ddots & \vdots \\
B_{\sigma_s} & \mathbf{0} & \cdots & \mathbf{0} & \tilde p_{\sigma_s}
\end{bmatrix}\! ,\, \hbox{where $\tilde p_i:= \bmat{w_i \!\! &  \!\! x_i \!\! &  \!\! y_i \!\! & \!\! z_i }^{\! T}$ for $i \in [n]$,}
$$
have non-vanishing leading coefficients. We see that $Z_{4,n} \subseteq \tin_{\prec}(J^B)$
by reasoning akin to that in the proof of Lemma \ref{lem:generic1}. The equality
$Z_{4,n}=\tin_{\prec}(J^B)$ is then immediate since
$Z_{4,n}$ is the generic initial ideal of $J^B$.
Hence, for any generic camera positions $A$, we can add a row to $A_i$ and
get $B_i$ that are ``sufficiently generic'' for Lemma \ref{lem:generators of Z}  (2).
This completes the proof.
~\end{proof}

\smallskip

\noindent
{\em Proof of Theorem \ref{thm:UGB}:}
Lemma~\ref{lem:generic2} and the proof of Lemma~\ref{lem:generic1} show 
that the maximal minors of the matrices $A_\sigma$ for $2 \leq |\sigma| \leq 4$
are a Gr\"obner basis of $J_A$ for the lexicographic term order. 
Each polynomial in that Gr\"obner basis is multilinear,
thus the initial monomials remain the same for any
term order satisfying $x_i \succ y_i \succ z_i$ for $i  = 1,2,\ldots,n$.
So, the minors form a Gr\"obner basis for that term order.
The set of minors is invariant under permuting 
$\{x_i,y_i,z_i\}$ for each $i$. 
Moreover, the genericity of $A$ implies that every monomial 
which can possibly appear in the support of a minor does so.
Hence, these minors form a universal Gr\"obner basis of $J_A$.
\qed

\begin{remark}
Computer vision experts have known for a long time that multiview 
varieties $V_A$ are defined set-theoretically by the above multilinear constraints of degree at most $4$.
We refer to work of Heyden and {\AA}str{\"o}m \cite{HA, Heyden}.
What is new here is that these constraints define $V_A$
in the strongest possible sense: they form a universal Gr\"obner basis
for the prime ideal $J_A$.
\end{remark}

The $n$ cameras are in {\em linearly general position} if no four focal points are coplanar and no three are collinear.
While the number of multilinear polynomials in our lex Gr\"obner basis of $J_A$ is
$\,\binom{n}{2} + 3 \binom{n}{3} + \binom{n}{4}$, far
fewer suffice to generate the ideal $J_A$ when $A$ is in linearly general position.

\begin{corollary}
If $A$ is in linearly general position then the ideal $J_A$ is minimally generated by $\,\binom{n}{2}$  bilinear and $\binom{n}{3}$ trilinear polynomials. 
\end{corollary}

\begin{proof}
This can be shown for $n \leq 4$ by a direct calculation.
Alternatively, these small cases are covered by transforming to the toric ideals
in Section 4.
First map the focal points of the cameras to the 
torus fixed focal points of the toric case, followed by multiplying each $A_i$ by a suitable 
$g_i \in \PGL(3,K)$.

Now let $n \geq 5$.
For any three cameras $i, j, k$, 
the maximal minors of (\ref{eq:drei}) are generated
by only one such maximal minor modulo
the three bilinear polynomials (\ref{eq:zwei}).
Likewise, for any four cameras $i$, $j$, $k$ and $l$,
the maximal minors of (\ref{eq:vier})
are generated by the trilinear and bilinear polynomials.
This implies that the resulting
$\binom{n}{2} + \binom{n}{3}$ polynomials generate $J_A$,
and, by restricting to two or three cameras, we see that they
minimally generate.
\end{proof}

\section{The Generic Initial Ideal}

We now focus on  combinatorial properties of our special monomial ideal
$$ M_n \quad = \quad \bigl\langle \,
 x_ix_j, \, x_iy_jy_k, \,y_iy_jy_ky_l \,\,:\,\, \forall \,\,i,j,k,l  \in [n] \,\, \textup{distinct}\bigr\rangle. $$
We refer to $M_n$ as the {\em generic initial ideal} in multiview geometry because it is the lex initial ideal of
any multiview ideal $J_A$ after a generic coordinate change 
via the group $G^n$ where $G = {\PGL}(3,K)$. Indeed, consider {\bf any} rank $3$ matrices
$A_1,A_2, \ldots,A_n \in K^{3 \times 4}$ with
pairwise distinct kernels $K \{f_i\}$. If $g = (g_1,g_2,\ldots,g_n) $ is
generic in $G^n$ then $g \circ A$ is generic in the sense that
all $4 \times 4$-minors of the matrix $ \bmat{ (g_1 A_1)^T \! \! & \!\! (g_2 A_2)^T \!\! & \!\!  \cdots \!\! & \!\! (g_n A_n)^T }$
are non-zero.
Thus, by the results of Section 2, $M_n$ is the initial ideal of $J_{g \circ A}$, or, using standard
commutative algebra lingo, $M_n$ is the generic initial ideal of $J_A$.

Since $M_n$ is a squarefree monomial ideal, it is radical. Hence $M_n$ is the intersection of its minimal primes,
which are generated by subsets of the variables $x_i$ and $y_j$.
We begin by computing this prime decomposition.

\begin{proposition} \label{prop:prime_decomposition}
The generic initial ideal $M_n$ is the irredundant intersection of 
$\binom{n}{3} + 2 \binom{n}{2}$ monomial primes. These are the monomial primes
$P_{ijk}$ and $Q_{ij}\subseteq K[x,y,z]$ defined below for any distinct indices $i,j,k\in[n]$:
\begin{itemize}
\item $P_{ijk}$ is generated by $x_1,\dots,x_n$  and all $y_l$ with $l\not\in\{i,j,k\}$,
\item $Q_{ij}$ is generated by all $x_l$ for $l\ne i$ and $y_l$ for $l\not\in\{i,j\}$. 
\end{itemize}
\end{proposition}

\proof
Let $L$ denote the intersection of all $P_{ijk}$ and $Q_{ij}$.
Each monomial generator of $M_n$ lies in $P_{ijk}$ and in $Q_{ij}$, so
$M_n \subseteq L$. For the reverse inclusion, we will show that 
 $V(M_n)$ is contained in $V(L) =  (\cup V(P_{ijk})) \cup (\cup V(Q_{ij}))$.

Let $(\tilde{x},\tilde{y}, \tilde{z})$ be any point in the variety $V(M_n)$. First suppose 
$\tilde{x}_i=0$ for all $i\in[n]$. Since $\tilde{y}_i \tilde{y}_j \tilde{y}_k \tilde{y}_l=0$ for distinct indices,
there are at most three indices $i,j,k$ such that $\tilde y_i$, $\tilde y_j$ and $\tilde y_k$ are nonzero. 
Hence $(\tilde{x},\tilde{y}, \tilde{z})  \in V(P_{ijk})$.

Next suppose $\tilde x_i \not = 0$. The index $i$ is unique because $x_i x_j \in M_n$ for all $j \not= i$.
Since $\tilde x_i \tilde y_j \tilde y_k=0$ for all $j,k\ne i$, we have $\tilde y_j \ne 0$
for at most one index $j\ne i$. These properties imply
$(\tilde{x},\tilde{y}, \tilde{z})  \in V(Q_{ij})$.
\endproof

We regard the monomial variety $V(M_n)$ as a threefold inside
the product of projective planes $(\PP^2)^n$.
If the focal points are distinct, $V_A$ has a Gr\"obner degeneration to the reducible threefold $V(M_n)$.
The irreducible components of $V(M_n)$ are
\begin{equation}
\label{eq:components} V(P_{ijk}) \,\simeq \,\PP^1 \times \PP^1 \times \PP^1 
\quad \hbox{and} \quad
V(Q_{ij}) \,\simeq \,\PP^2 \times  \PP^1. 
\end{equation}
We find it convenient to regard $(\PP^2)^n$ as a toric variety, so as to
identify it with its polytope $(\Delta_2)^n$, a direct product of triangles.
The components in (\ref{eq:components}) are $3$-dimensional boundary
strata of $(\PP^2)^n$, and we identify them with faces of $(\Delta_2)^n$.
The corresponding $3$-dimensional polytopes are  the {\em $3$-cube}
and the {\em triangular prism}. The following three examples illustrate this view.

\begin{figure}
\includegraphics[width=0.56\linewidth]{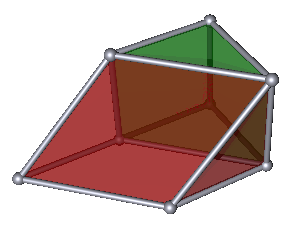}
\caption{The variety of the generic initial ideal $M_2$ seen as two adjacent facets
of the $4$-dimensional polytope $\Delta_2 \times \Delta_2$.}
\label{V2_Z_Monomial_Figure}
\end{figure}

\begin{example}{\rm [Two cameras $(n=2)$] \ }
The variety of $\,M_2 = \langle x_1 \rangle \,\cap \, \langle x_2 \rangle \,$
is a hypersurface in $\PP^2 \times \PP^2$.
The two components are triangular prisms $\PP^2 \times \PP^1$,
which are glued along a common square $\PP^1 \times \PP^1$, as shown in 
Figure~\ref{V2_Z_Monomial_Figure}. \qed
\end{example}

\begin{example}{\rm [Three cameras $(n=3)$] \ }  \label{ex:M_3}
The variety of $M_3 $ is a threefold in $\PP^2 \times \PP^2 \times \PP^2$.
Its seven components are given by the prime decomposition
$$ 
\begin{matrix} M_3 \quad = & \quad
 \langle x_1, x_2, y_1 \rangle \, \cap \,
  \langle x_1, x_2, y_2 \rangle \, \cap \,
    \langle x_1, x_3, y_1 \rangle  & \\ 
&    \, \cap \,\,
      \langle x_1, x_3, y_3 \rangle \, \cap \,
   \langle x_2, x_3, y_2 \rangle \, \cap \,
      \langle x_2, x_3, y_3 \rangle & \!\! \cap \,\,
      \langle x_1, x_2, x_3 \rangle .
\end{matrix}
       $$
The last component is a cube $\PP^1 \times \PP^1 \times \PP^1$,
and the other six components are triangular prisms $\PP^2 \times \PP^1$.
These are glued in pairs along three of the six faces of the cube.
For instance, the two triangular prisms $V(x_1,x_2,y_1)$
and $V(x_1,x_3,y_1)$ intersect the cube $V(x_1,x_2,x_3)$
in the common square face $V(x_1,x_2,x_3,y_1)$  $\simeq \PP^1\times \PP^1$.
This polyhedral complex lives in the boundary of $(\Delta_2)^3$,
and it shown in  Figure~\ref{V3_Z_Monomial_Figure}. 
Compare this picture with Figure \ref{V3_J8_Blowup}.
\qed
 \end{example}

\begin{figure}
\includegraphics[width=0.43\linewidth]{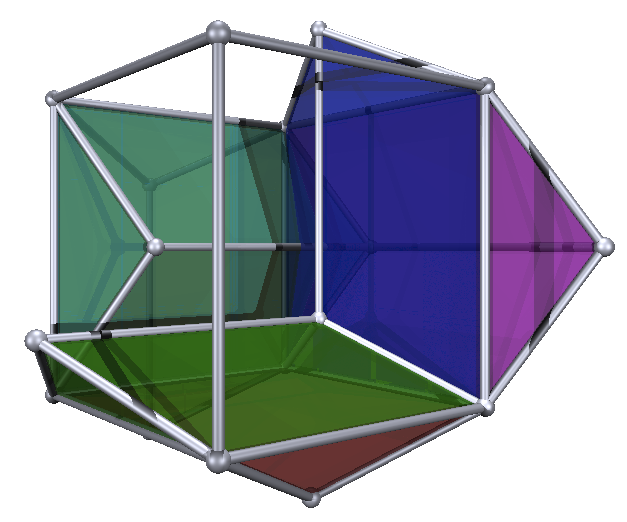} 
\!\!\!\!\!
\includegraphics[width=0.55\linewidth]{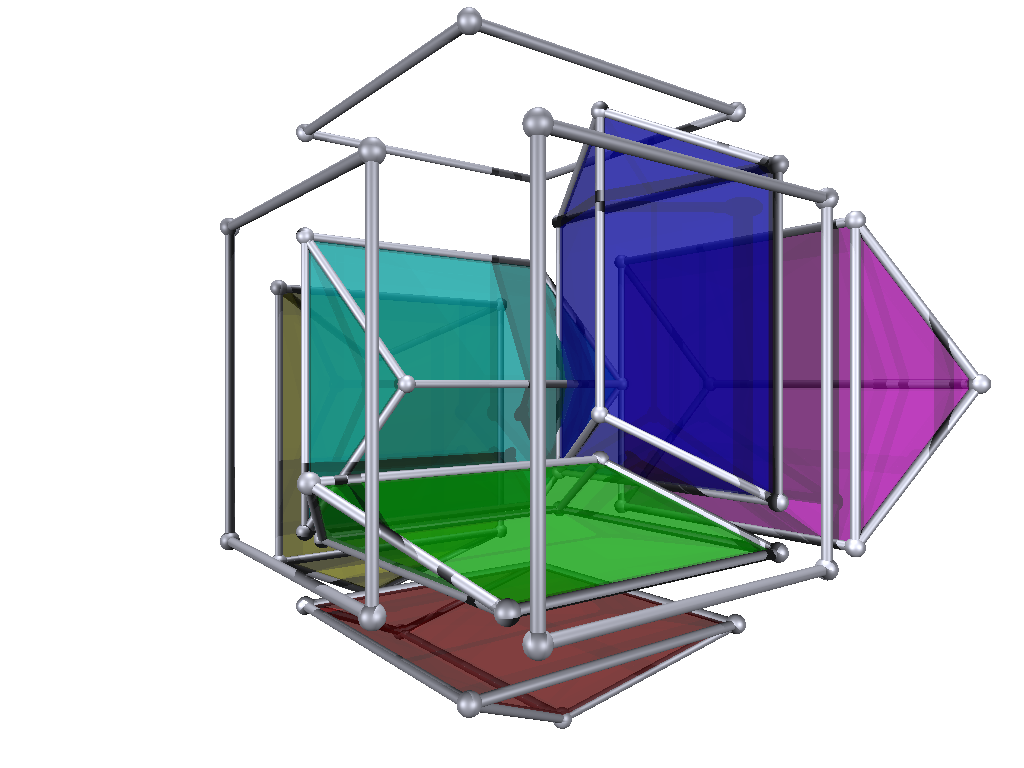}
\caption{The monomial variety $V(M_3)$ as a subcomplex of $(\Delta_2)^3$.}
\label{V3_Z_Monomial_Figure}
\end{figure}

\begin{example} {\rm [Four cameras $(n = 4)$] } \label{ex:M_4}
The variety $V(M_4) $ is a threefold in $(\PP^2)^4$,
regarded as a $3$-dimensional subcomplex
in the boundary of the $8$-dimensional polytope $(\Delta_2)^4$.
It consists of four cubes and twelve triangular prisms.
The cubes share a common vertex, any two cubes
intersect in a square, and each of the six squares
is adjacent to two triangular prisms. \qed
\end{example}

From the prime decomposition in Proposition \ref{prop:prime_decomposition}
we can read off the {\em multidegree} \cite[\S 8.5]{MS} of the ideal $M_n$.
Here and in what follows, we use
 the natural $\ZZ^n$-grading on $K[x,y,z]$ given by
$\deg(x_i) = \deg(y_i) = \deg(z_i) = e_i$.
Each multiview ideal $J_A$ is homogeneous with respect to this
$\ZZ^n$-grading.

\begin{corollary}
The multidegree of the generic initial ideal $M_n $ is equal to
\begin{equation}
\label{eq:multidegree} \mathcal{C}\bigl(K[x,y,z]/M_n; {\bf t} )\bigr) \,\,\, = \, \,\,
t_1^2 t_2^2 \cdots t_n^2 \cdot \left( \sum_{1 \leq i < j < k \leq n} \! \frac{1}{t_i t_j t_k} \,+ \!
\sum_{1 \leq i,j \leq n} \! \frac{1}{t_i^2 t_j}\, \right)
\end{equation}
\end{corollary}

A more refined analysis also yields the Hilbert function in the $\ZZ^n$-grading.

\begin{theorem} \label{thm:hilbertfct}
The multigraded Hilbert function of $K[x,y,z]/M_n$ equals
\begin{equation}
\label{Vn_Hilbert_function}
\NN^n \,\to \,\NN,\ (u_1,\ldots,u_n)\,\mapsto \,{u_1+\cdots+u_n+3\choose 3}-\sum_{i=1}^n{u_i+2\choose 3}.
\end{equation}
\end{theorem}

\proof Fix  $u\in\NN^n$. A $K$-basis $\mathfrak{B}_u$ for $(K[x,y,z]/M_n)_u$ is given by all monomials $x^ay^bz^c \not \in M_n$ such that $a+b+c=u$.  Therefore, either (i) $a=0$ and at most three components 
of $b$ are non-zero; or (ii) $a\ne 0$, in which case only one $a_i$ can be non-zero and $b_j \not= 0$ for at most 
one $j \in [n] \backslash \{i\}$.

We shall count the monomials in $\mathfrak{B}_u$. Monomials of type (i) look like $y^bz^c$, with at most three nonzero entries in $b$. Also, $b$ determines $c$ since $c_i = u_i - b_i$ for all $i \in [n]$, and so we count the number of possibilities for $y^b$.
There are $u_i$ choices for $b_i \ne 0$, and thus $U := u_1+\cdots+u_n$ many monomials in the set ${\mathcal Y} := \{y_i^{b_i} \,:\, 1 \leq b_i \leq u_i, \,i=1,\ldots,n\}$. The factor $y^b$ in $y^bz^c$ is the product of $0$, $1$, $2$ or $3$ monomials from ${\mathcal Y}$ 
with distinct subscripts.

To resolve over-counting, consider a fixed index $i$.  There are ${u_i\choose 2}$ ways of choosing two monomials from ${\mathcal Y}$ with subscript $i$ and ${u_i\choose 3}$ ways of choosing three monomials from ${\mathcal Y}$ with subscript $i$. Also, there are ${u_i\choose 2}(U- u_i)$ ways of choosing two monomials from ${\mathcal Y}$ with subscript $i$ and a third monomial with a different subscript. Hence, the number of choices for $y^b$ in $y^bz^c$ is

\begin{small}
$${U\choose 0}+{U\choose 1}+\left[{U\choose 2} - \sum_{i=1}^n \! {u_i\choose 2}\right]+
\left[{U\choose 3} -\sum_{i=1}^n \! {u_i\choose 3} -U\sum_{i=1}^n \! {u_i\choose 2}+\sum_{i=1}^nu_i{u_i\choose 2}\right]\!.$$
\end{small}

For case (ii) we count all monomials $x^ay^bz^c\in\mathfrak{B}_u$ with $a_i\ne 0$ and all other $a_j=0$. 
It suffices to count the choices for the factor $x^ay^b$. For fixed $i$, there are ${u_i+1\choose 2}$ monomials of the form 
$x_i^{a_i}y_i^{b_i}$ with $a_i+b_i \leq u_i$ and $a_i \geq 1$.  Such a monomial may be multiplied with 
$y_j^{b_j}$ such that $j \ne i$ and $0 \leq b_j \leq u_j$.  This amounts to choosing zero or one monomial from ${\mathcal Y} \backslash \{y_i, y_i^2, 
\ldots, y_i^{u_i} \}$ for which there are $1+U-u_i$ choices. Hence, there are 
$$[1+U]\sum_{i=1}^n {u_i+1\choose 2} \,-\, \sum_{i=1}^nu_i{u_i+1\choose 2}$$
monomials in $\mathfrak{B}_u$ of type (ii).
Adding the two expressions, we get 
\begin{smaller}
\begin{align*}
|\mathfrak{B}_u| 
&=  1+U+{U\choose 2}+{U\choose 3}+(1+U)\sum_{i=1}^n{u_i\choose 1} -\sum_{i=1}^n u_i{u_i\choose 1} -\sum_{i=1}^n {u_i\choose 3}\\
&= 1+U+{U\choose 2}+{U\choose 3}+(1+U)U - \sum_{i=1}^n{u_i+2\choose 3}\\
&= {U+3\choose 3}-\sum_{i=1}^n{u_i+2\choose 3}.
\end{align*}
\end{smaller}
\vskip -0.9cm
\endproof

\smallskip

Our analysis of $M_n$ has the following implication for the  multiview ideals $J_A$.
Note that these are $\ZZ^n$-homogeneous for any camera configuration $A$.

 \begin{theorem} \label{thm:multiview ideals on Hilbert scheme}
For an $n$-tuple of camera matrices $A = (A_1,\ldots, A_n)$ with $\textup{rank}(A_i)=3$ for each $i$, the
multiview ideal $J_A$ has the Hilbert function (\ref{Vn_Hilbert_function}) 
if and only if the focal points of the $n$ cameras are pairwise distinct.
 \end{theorem}

\begin{proof}
The if-direction follows from the argument in the first paragraph of this section.
If the $n$ camera positions $f_i = {\rm ker}(A_i)$  are distinct in $\PP^3$
then $M_n$ is the generic initial ideal of $J_A$, and hence both ideals have
 the same $\ZZ^n$-graded Hilbert function.
 For the only-if-direction we shall use:
 \begin{equation} \label{lem:multiview ideal stays same}
\hbox{If $Q \in \PGL(4,K)$ and $AQ := (A_1Q, \ldots, A_nQ)$, then $J_A = J_{AQ}$.}
\end{equation}
This holds because $Q$ defines an isomorphism on $\PP^3 $
and hence $\phi_A$ as in (\ref{eq:phiA}) has the same image in $(\PP^2)^n$ as
$\phi_{AQ}$.

Suppose first that $n=2$ and $A_1$ and $A_2$ have the same focal point
and hence the same (three-dimensional) rowspace $W$. 
We can map $W$ to the hyperplane $\{x_1 = 0\}$ by some
$Q \in \PGL(4,K)$, and  (\ref{lem:multiview ideal stays same})
ensures that
 $J_A = J_{AQ}$. 
 Thus we may assume that $ A_1 = \left[ \begin{array}{ll} \bf{0} & C_1 \end{array} \right]$ and $A_2 = \left[ \begin{array}{ll} \bf{0} & C_2 \end{array} \right]$ where $C_1$ and $C_2$ are invertible matrices and $\bf{0}$ is a column of zeros. Choosing 
$f_1 = f_2 = (1,0,0,0)$ as the top row of $B_1$ and $B_2$ (as in Section 2), we have 
$$ B_1^{-1} = \left[ \begin{array}{cc} 1 & \bf{0} \\ \bf{0} & C_1^{-1} \end{array} \right], \,\,\,B_2^{-1} = \left[ \begin{array}{cc} 1 & \bf{0} \\ \bf{0} & C_2^{-1} \end{array} \right].$$

The ideal $J^B$ is generated by the $2\times2$ minors of the matrix (\ref{eq:inverseB}) which is 
$$D = \left[ \begin{array}{cc} w_1 & w_2 \\
p_1(x_1,y_1,z_1) & q_1(x_2,y_2,z_2) \\ p_2(x_1,y_1,z_1)  & q_2(x_2,y_2,z_2) \\ p_3(x_1,y_1,z_1)  & q_3(x_2,y_2,z_2)  \end{array} \right]$$
where the $p_i$'s and $q_i$'s are linear polynomials. The ideal $I$ generated by the $2 \times 2$ minors of the submatrix of $D$ obtained by deleting the top row lies on the Hilbert scheme $H_{3,2}$ from \cite{CS} and hence $K[x,y,z]/I$ has Hilbert function 
$$ \mathbb N^2 \to \mathbb N, \,\,\,(u_1,u_2) \mapsto \left( \begin{array}{c} 
u_1+u_2+2 \\ 2 \end{array} \right).$$
For $(u_1,u_2) = (1,1)$, this has value
$6$. Since $I \subseteq J_A = J^B \cap K[x,y,z] $, the Hilbert function
of $ K[x,y,z]/J_A$ has value $\leq 6$, while (\ref{Vn_Hilbert_function}) evaluates to $8$.

If $n > 2$, we may assume without loss of generality that $A_1$ and $A_2$ have the same rowspace. 
The argument for $n=2$ shows that $J_A = J^B \cap K[x,y,z] \supseteq I$. The Hilbert function value of $K[x,y,z]/J_A$ in degree $e_1+e_2$ is again $8$, while the Hilbert function value of $K[x,y,z]/I$ in degree $e_1+e_2$ coincides with 
the value $6$ for $K[x_1,y_1,z_1,x_2,y_2,z_2]/I$. So we again conclude that 
$K[x,y,z]/J_A$ does not have Hilbert function (\ref{Vn_Hilbert_function}).
 \end{proof}

For $G = \PGL(3,K)$, the product $G^n$ acts on $K[x,y,z]$ by left-multiplication
$$ (g_1,\ldots,g_n) \cdot  \left[ \begin{array}{c} x_i \\ y_i \\ z_i \end{array} \right] 
\,\,=\,\, \, g_i  \left[ \begin{array}{c} x_i \\ y_i \\ z_i \end{array} \right].$$
An ideal $I$ in $K[x,y,z]$ is said to be {\em Borel-fixed} if it is fixed
under the induced action of $\B^n$ where $\B$ is the subgroup of lower triangular
matrices in $G$.

\begin{proposition}
The generic initial ideal $M_n$ is the unique ideal  in $K[x,y,z]$
that is Borel-fixed and has the Hilbert function (\ref{Vn_Hilbert_function})
in the $\ZZ^n$-grading.
\end{proposition}

\proof
The proof is analagous to that of 
\cite[Theorem 2.1]{CS}, where $Z_{d,n}$ plays the role of $M_n$.
The ideal $M_n$ is Borel-fixed because it is a generic initial ideal.
The same approach as  in \cite[\S 15.9.2]{E} can be used to prove this fact.

The multidegree of any $\ZZ^n$-graded ideal is determined by its Hilbert series \cite[Claim 8.54]{MS}.
Thus any ideal $I$ with Hilbert function (\ref{Vn_Hilbert_function}) has multidegree (\ref{eq:multidegree}).
Let $I$ be such a Borel-fixed ideal. This is a monomial ideal.

Each maximum-dimensional associated prime $P$ of  $I$ has multidegree either
$t_1^2t_2^2\cdots t_n^2/(t_it_jt_k)$ or $t_1^2t_2^2\cdots t_n^2/(t_i^2t_j)$, by \cite[Theorem 8.53]{MS}.
In the first case $P$ is generated by $2n-3$ indeterminates, one associated with each 
of the three cameras $i,j,k$ and two each from the other $n-3$ cameras.
Borel-fixedness of $I$ tells us that the generators indexed by each camera must be the most
expensive variables with respect to the order $\prec$. Hence $P=P_{ijk}$.
Similarly, $P=Q_{ij}$ in the case when $P$ has multidegree  $t_1^2t_2^2\cdots t_n^2/(t_i^2t_j)$.

Every prime component of $M_n$ is among the minimal associated primes of $I$.
This yields the containments $I\subseteq \sqrt{I}\subseteq M_n$.
Since $I$ and $M_n$ have the same $\ZZ^n$-graded Hilbert function,
the equality $I=M_n$ holds.
\endproof

The {\em Stanley-Reisner complex} of a squarefree monomial ideal $M$ in a polynomial ring
$K[t_1,\ldots,t_s]$ is the simplicial complex on $\{1, \ldots, s \}$ whose facets are the sets $[s] \backslash \sigma$ where
$P_\sigma := \{ t_i \,:\, i \in \sigma \}$ is a minimal prime of $M$. A {\em shelling} of a simplicial complex
is an ordering  $F_1,F_2, \ldots, F_q$
of its facets such that, for each $1 < j \leq q$, there exists a unique
minimal face of $F_j$ (with respect to inclusion) among the faces of $F_j$ that are not faces of some earlier facet $F_i$, $i < j$; see \cite[Definition 2.1]{Stanley}.
If the Stanley-Reisner complex of $M$ is shellable, then
$K[t_1, \ldots, t_s]/M$ is Cohen-Macaulay \cite[Theorem 2.5]{Stanley}.

\begin{proposition}
The Stanley-Reisner complex of the generic initial ideal $M_n$ is shellable.
Hence the quotient ring $K[x,y,z]/M_n$ is Cohen-Macaulay.
\end{proposition}

\proof
This proof is similar to that for $Z_{d,n}$ given in \cite[Corollary 2.6]{CS}.
 Let $\Delta_n$ denote the Stanley-Reisner complex of the ideal $M_n$.
By Proposition~\ref{prop:prime_decomposition}, there are two types of minimal primes for $M_n$,
namely $P_{ijk}$ and $Q_{ij}$, which we describe uniformly as follows. Let $P = (p_{ij})$ be the $3 \times n$ matrix whose $i$th column is $[x_i \,\, y_i \,\, z_i]^T$. For $u \in \{0,1,2\}^n$ define
$P_u := \langle p_{ij} \,:\, i \leq u_j, \, 1 \leq j \leq n \rangle$. Then the minimal primes $P_{ijk}$ of $M_n$ are precisely the primes $P_u$ as $u$ varies over all vectors with three coordinates equal to one and the rest equal to two, and the minimal primes $Q_{ij}$ are those $P_u$ where $u$ has one coordinate equal to zero, one coordinate equal to one and the rest equal to two. The facet of $\Delta_n$ corresponding to the minimal prime $P_u$ is then $F_u := \{ p_{ij} \,:\, u_j < i \leq 3, \, 1 \leq j \leq n \}$.
We claim that the ordering of the facets $F_u$ induced by ordering the $u$'s lexicographically starting with $(0,1,2,2, \ldots, 2)$ and ending with $(2,2, \ldots, 2,1,0)$ is a shelling of $\Delta_n$.

Consider the face  $\eta_u := \{ p_{ij} \,: \, j > 1, i = u_j+1 \leq 2\}$ of the facet $F_u$.
We will prove that $\eta_u$ is the unique minimal one among the faces of $F_u$ that have not appeared in a facet $F_{u'}$ for $u' < u$.  Suppose $G$ is a face of $F_u$ that does not contain $\eta_u$. Pick an element $p_{u_j+1,j} \in \eta_u \backslash G$. Then $j > 1$, $u_j \leq 1$ and so if $F_u$ is not the first facet in the ordering, then there exists $i < j$ such that $u_i > 0$ because $u >
(0,1,2,2,\ldots,2)$ and of the form described above. Pick $i$ such that $i < j$ and $u_i > 0$ and consider
$F_{u+e_j-e_i} = F_u \backslash \{p_{u_j+1,j} \} \cup \{p_{u_i,i} \}$. Then $u+e_j-e_i < u$ and
$G$ is a face of $F_{u+e_j-e_i}$. Conversely, suppose $G$ is a face of $F_u$ that is also a face of $F_{u'}$ where $u' < u$. Since $\sum u'_j = \sum u_j$, there exists some $j > 1$ such that $u'_j > u_j$. Therefore, $G$ does not contain $p_{u_j+1,j}$ which belongs to $\eta_u$. Therefore, $\eta_u$ is not contained in $G$.
\endproof

\section{A Toric Perspective}

In this section we examine multiview ideals $J_A$ that are toric.
For an introduction to toric ideals we refer the reader to \cite{GBCP}.
We now
assume that, for each camera $i$, each of the 
four torus fixed points in $\PP^3$ either is the camera position
or is mapped to a torus fixed point in $\PP^2$. 
This implies  $n \leq 4$. We 
fix $n=4$ and  $f_i = e_i$ for $i=1,2,3,4$. Up to permuting and rescaling columns,
our assumption implies that the configuration $A$ equals
$$
\begin{small}
 A_1 = \begin{bmatrix} 0 \! & \! 1 \! & \! 0 \! & \! 0 \\
                                           0 \! & \! 0 \! & \! 1 \! & \! 0 \\
                                           0 \! & \! 0 \! & \! 0 \! & \! 1  \end{bmatrix} \! ,\,\,
 A_2 = \begin{bmatrix} 1 \! & \! 0 \! & \! 0 \! & \! 0 \\
                                           0 \! & \! 0 \! & \! 1 \! & \! 0 \\
                                           0 \! & \! 0 \! & \! 0 \! & \! 1  \end{bmatrix} \!,\,\,
A_3 = \begin{bmatrix} 1 \! & \! 0 \! & \! 0 \! & \! 0 \\
                                           0 \! & \! 1 \! & \! 0 \! & \! 0 \\
                                           0 \! & \! 0 \! & \! 0 \! & \! 1  \end{bmatrix} \!,\,\,
A_4 = \begin{bmatrix} 1 \! & \! 0 \! & \! 0 \! & \! 0 \\
                                           0 \! & \! 1 \! & \! 0 \! & \! 0 \\
                                           0 \! & \! 0 \! & \! 1 \! & \! 0  \end{bmatrix} \! .
\end{small}
$$
For this camera configuration, the multiview ideal $J_A$ is indeed a toric ideal:

\begin{proposition} \label{prop:toricideal}
The ideal $J_A$ is obtained  by  eliminating the diagonal unknowns $w_1$, $w_2$, $w_3$ and $w_4$
from the ideal
of $2 \times 2$-minors of the $4 \times 4$-matrix
\begin{equation}
\label{eq:fourbyfour}
\begin{pmatrix}
w_1   & x_2 & x_3 & x_4 \\
x_1 & w_2   & y_3 & y_4 \\
y_1 & y_2 & w_3   & z_4 \\
z_1 & z_2 & z_3 & w_4
\end{pmatrix}.
\end{equation}
This toric ideal is minimally generated by six quadrics and four cubics:
\begin{small}
$$ \! \begin{matrix} J_A \,=\, \langle
y_1 y_4{-}x_1 z_4, y_3 x_4{-}x_3 y_4, y_2 x_4{-}x_2 z_4, z_1 y_3{-}x_1 z_3, z_2 x_3{-}x_2 z_3,
     z_1 y_2{-}y_1 z_2 ,\\ \qquad \quad
     y_2 z_3 y_4-z_2 y_3 z_4, \,y_1 z_3 x_4-z_1 x_3 z_4,\, x_1 z_2 x_4-z_1 x_2 y_4,\,
     x_1 y_2 x_3-y_1 x_2 y_3 \rangle
     \end{matrix}
$$
\end{small}
\end{proposition}

\begin{proof}
We extend $A_i$ to a $4 \times 4$-matrix $B_i$ as in Section~2 by adding the
row $b_i = e_i^T$. The $B_i$'s are then all permutation matrices,
and  the matrix in (\ref{eq:inverseB}) equals the matrix in (\ref{eq:fourbyfour}).
The ideal $J^B$ is generated by the $2 \times 2$ minors of that matrix of unknowns.
The multiview ideal is $J_A = J^B \cap K[x,y,z]$. We find the listed
binomial generators  by performing  the elimination with a
computer algebra package such as {\tt Macaulay2}.
 Toric ideals are precisely those prime
ideals generated by binomials and hence $J_A$ is a toric ideal.
\end{proof}

\begin{remark}
The {\em normalized coordinate system in multiview geometry}
proposed by Heyden and  {\AA}str{\"o}m \cite{HA} is different from ours
and does not lead to toric varieties. Indeed, if one uses the camera matrices in
\cite[\S 2.3]{HA}, then  $J_A$ is also generated by six quadrics
and four cubics, but seven of the ten generators are not binomials.
One of the cubic generators has six terms. \qed
\end{remark}

In commutative algebra, it is customary to represent
toric ideals by  integer matrices. Given $\mathcal{A} \in \mathbb N^{p \times q}$ with columns
$a_1, \ldots, a_q$, the {\em toric ideal} of $\mathcal{A}$ is
$$ I_\mathcal{A} \,\, := \,\, \langle t^u - t^v \,:\, \mathcal{A}u 
\,\,= \,\, \mathcal{A}v, \, u, v \in \mathbb N^q \rangle \,\, \subset \,\, K[t] \,:= \,K[t_1, \ldots, t_q], $$
where $t^u$ represents the monomial $t_1^{u_1}t_2^{u_2} \cdots t_q^{u_q}$.
If $\mathcal{A'}$ is the submatrix of $\mathcal A$ obtained by deleting the columns indexed by $j_1, \ldots, j_s$ for some $s < q$, then the toric ideal $I_{\mathcal{A'}}$ equals the elimination ideal $I_\mathcal{A} \cap K[t_j \,:\,
j \not \in \{j_1, \ldots, j_s\}]$; see \cite[Prop.~4.13 (a)]{GBCP}. The integer matrix $\mathcal{A}$
for our toric multiview ideal $J_A$ in Proposition \ref{prop:toricideal}
is the following {\em Cayley matrix}  of format $8 \times 12$:
$$
\mathcal{A} \,\,\, = \,\,\,
\begin{bmatrix}
 A_1^T &  A_2^T &  A_3^T &  A_4^T \\
{\bf 1} & {\bf 0} & {\bf 0} & {\bf 0} \\
{\bf 0} & {\bf 1} & {\bf 0} & {\bf 0} \\
{\bf 0} & {\bf 0} & {\bf 1} & {\bf 0} \\
{\bf 0} & {\bf 0} & {\bf 0} & {\bf 1}
\end{bmatrix}
$$
where ${\bf 1} = [1 \,1\,1 ]$ and ${\bf 0} = [0\,0\,0 ]$.
This matrix $\mathcal{A}$ is obtained from the following $8 \times 16$ matrix
by deleting columns $1, 6, 11$ and $16$:
\begin{equation}
\label{eq:transportation}
\begin{bmatrix}
I_4 & I_4 & I_4 & I_4 \\
{\bf 1} & {\bf 0} & {\bf 0} & {\bf 0} \\
{\bf 0} & {\bf 1} & {\bf 0} & {\bf 0} \\
{\bf 0} & {\bf 0} & {\bf 1} & {\bf 0} \\
{\bf 0} & {\bf 0} & {\bf 0} & {\bf 1}
\end{bmatrix}
\end{equation}
 The vectors ${\bf 1}$ and ${\bf 0}$ now have length four, $I_4$ is the $4 \times 4$ identity matrix and we assume that the columns of (\ref{eq:transportation}) are indexed by $$w_1, x_1, y_1,z_1, x_2,w_2,y_2,z_2,x_3,y_3,w_3,z_3,x_4,y_4,z_4,w_4.$$ The matrix 
 (\ref{eq:transportation}) represents the direct product of two tetrahedra, 
 and its toric ideal is known  (by  \cite[Prop. 5.4]{GBCP})
 to be  generated by the $2 \times 2$ minors of (\ref{eq:fourbyfour}).
        Its elimination ideal in the ring $K[x,y,z]$ is $I_\mathcal{A}$, and hence
        $J_A = I_\mathcal{A}$.

        \begin{figure}
\includegraphics[width=0.45\linewidth]{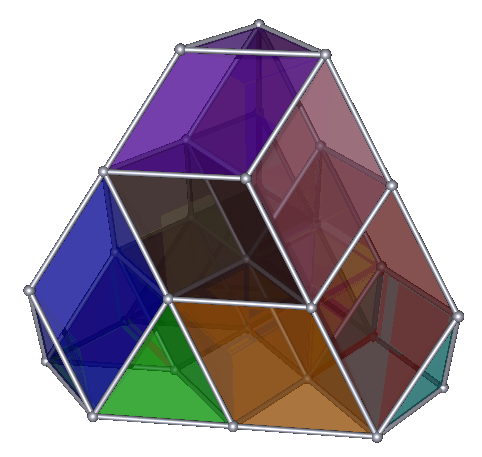}\hfill
\includegraphics[width=0.55\linewidth]{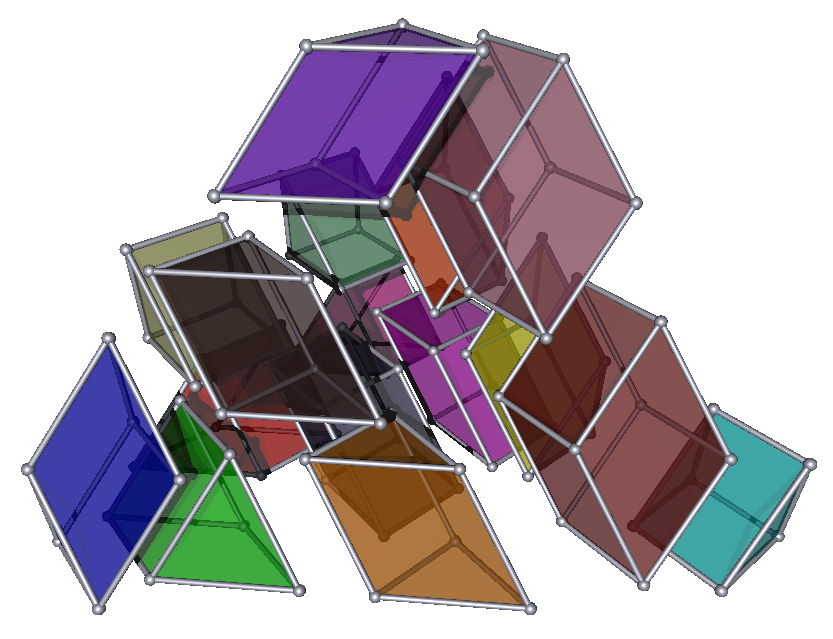}
\caption{
Initial monomial ideals of the toric multiview variety correspond to mixed subdivisions of
the truncated tetrahedron $P$. These have $4$ cubes and $12$ triangular prisms.}
\label{V4_Toric_Blowup}
\end{figure}

The matrix $\mathcal{A}$ has rank $7$ and its columns determine a
 $6$-dimensional polytope ${\rm conv}(\mathcal{A})$
  with $12$ vertices.
 The normalized volume of $ {\rm conv}(\mathcal{A})$ equals $16$, and this
is the degree of the $6$-dimensional projective toric variety in $\PP^{11}$ defined by $J_A$.
In our context, we don't care for the $6$-dimensional variety in $\PP^{11}$ 
but we are interested in the threefold in
       $\PP^2 {\times} \PP^2 {\times} \PP^2 {\times} \PP^2$ cut out by $J_A$.
       To study this combinatorially,   we apply the {\em Cayley trick}. This means we
 replace the $6$-dimensional polytope
${\rm conv}(\mathcal{A})$ by the $3$-dimensional polytope
$$  P \,\, = \,\, {\rm conv}(A_1^T) + {\rm conv}(A_2^T) + {\rm conv}(A_3^T) + {\rm conv}(A_4^T) . $$
This is the Minkowski sum of the four triangles that form the facets of the standard tetrahedron.
Equivalently, $P$ is the scaled tetrahedron $4 \Delta_3$ with its vertices sliced off.
Triangulations of $\mathcal{A}$ correspond to mixed subdivisions of $P$.
Each $6$-simplex in $\mathcal{A}$ becomes a cube or a triangular prism in $P$.
Each mixed subdivision has  four cubes $\PP^1 \times \PP^1 \times \PP^1$ 
and twelve triangular prisms $\PP^2 \times \PP^1$.
Such a mixed subdivision of $P$ is shown in Figure \ref{V4_Toric_Blowup}.
Note the similarities and differences relative to the complex $V(M_4)$ in Example \ref{ex:M_4}.

\smallskip

We worked out a complete classification of all mixed subdivisions of $P$:

\begin{theorem} \label{thm:1068}
The truncated tetrahedron $P$ has $1068$ mixed subdivisions, one
for each triangulation of the Cayley polytope ${\rm conv}(\mathcal{A})$.
Precisely $1002$ of the $1068$ triangulations are regular.
The regular triangulations form $48$ symmetry classes, and the
non-regular triangulations form $7$ symmetry classes.
\end{theorem}

We offer a brief discussion of this result and how it was obtained.
Using the software {\tt Gfan} \cite{Gfan}, we found that $I_\mathcal{A}$ has 
1002 distinct monomial initial ideals. These ideals fall into 48 symmetry classes under the 
natural action of $(S_3)^4 \rtimes S_4$ 
on $K[x,y,z]$ where the $i$-th copy of $S_3$ permutes the variables $x_i, 
y_i,z_i$, and $S_4$ permutes the labels of the cameras.
The matrix $\mathcal{A}$ being unimodular, each initial ideal of $I_\mathcal{A}$
is squarefree and  each triangulation of $\mathcal{A}$ is unimodular.
To calculate all non-regular triangulations, we used the
bijection between triangulations and $\mathcal{A}$-graded monomial ideals 
 in \cite[Lemma 10.14]{GBCP}. Namely, we ran a second
computation using the software package {\tt CaTS} \cite{CaTS}
that lists all $\mathcal{A}$-graded monomials ideals, and we
found their number to be  $1068$, and hence $\mathcal{A}$ has 66 non-regular triangulations.

 \begin{figure}
\includegraphics[width=0.36\linewidth]{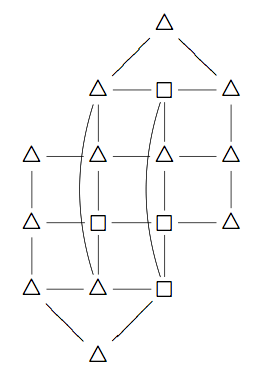}
\vskip -0.6cm
\caption{The dual graph of the mixed subdivision given by  $Y_1$.}
\label{fig:dualgraph12gens}
\end{figure}

The $48$ distinct initial monomial ideals of the toric multiview ideal $J_A$
can be distinguished by various invariants. First, their
numbers of generators range from $12$ to $15$.
  There is precisely one initial ideal with $12$ generators:
    \begin{align*}
Y_1  \,\,\, = \,\,\,  & \langle \,
 y_1z_2, z_1 y_3,  x_1z_4,  z_2 x_3, y_2 x_4, x_3y_4, \\  & \,\,\,\,
  x_1 y_2 x_3, z_1 y_2 x_3, x_1z_2 x_4, 
 z_1 x_3 z_4, z_2 y_3 x_4, z_2 y_3 z_4\,
\rangle.
 \end{align*}
At the other extreme, there are two classes of initial ideals with $15$ generators.
These are the only classes having quartic generators, as all ideals with $\leq 14$ generators
require only quadrics and cubics. A representative is
\begin{align*}
Y_2 \,\,\, = \,\,\, & \langle \,
z_1 y_2, x_1 z_3, x_1 z_4, x_2 z_3, y_2 x_4, y_3 x_4,
\, y_1 z_2 x_3 y_4, \\ & \,\,\,\,
x_1 y_2 x_3, x_1 z_2 x_3,  
x_1 z_2 x_4, x_4 z_2 y_1, 
y_1 z_3 x_4, y_1 z_3 y_4, 
y_2 x_3 y_4, y_2 z_3 y_4\,
\rangle.
\end{align*}
All non-regular $\mathcal{A}$-graded monomial ideal have $14$ generators.
One of them~is
\begin{align*}
Y_3 \,\,\, = \,\,\, & \langle \,
z_1y_2, z_1y_3, x_1z_4, x_2z_3, x_2z_4, 
y_3x_4,\, x_1y_2z_3, y_1x_2y_3, \\ &  \,\,\,\,
x_1y_2x_4, x_1z_2x_4, 
x_1z_3x_4, y_1z_3x_4, 
y_2z_3x_4, y_2z_3y_4\,
\rangle.
\end{align*}

A more refined combinatorial invariant of the $55$ types
is the dual graph of the mixed subdivision of $P$. The $16$ vertices of this
graph are labeled with squares and triangles to denote cubes and triangular prisms respectively,
and edges represent common facets.
The graph for $Y_1$ is shown in Figure~\ref{fig:dualgraph12gens}.

For complete information on the classification in
Theorem \ref{thm:1068} see the website \
\texttt{www.math.washington.edu/$\sim$aholtc/HilbertScheme}.

\smallskip

That website also contains the same information for the
toric multiview variey in the easier case of $n=3$ cameras. Taking
 $A_1, A_2$ and $A_3$ as camera matrices, the
corresponding Cayley matrix has format $7 \times 9$ and rank $6$:
$$
\mathcal{A} \,\,\, = \,\,\,
\begin{bmatrix}
 A_1^T & A_2^T &  A_3^T  \\
{\bf 1} & {\bf 0} & {\bf 0} \\
{\bf 0} & {\bf 1} & {\bf 0} \\ 
{\bf 0} & {\bf 0} & {\bf 1} 
\end{bmatrix}
\,\,\, = \,\,\,
\begin{bmatrix}
\,0 & 0 & 0 & 1 & 0 & 0 & 1 & 0 & 0 \,\\ 
\,1 & 0 & 0 & 0 & 0 & 0 & 0 & 1 & 0 \,\\
\,0 & 1 & 0 & 0 & 1 & 0 & 0 & 0 & 0 \,\\
\,0 & 0 & 1 & 0 & 0 & 1 & 0 & 0 & 1 \,\\
\, 1& 1 & 1 & 0 & 0& 0 & 0 & 0 & 0 \,\\
\, 0& 0 & 0 & 1 & 1& 1 & 0 & 0 & 0 \,\\
\, 0& 0 & 0 & 0 & 0& 0 & 1 & 1 & 1 \,
\end{bmatrix}
$$
This is the transpose of the matrix $A_{\{123\}}$ in (\ref{eq:drei}) when
evaluated at $x_1 = y_1 = \cdots = z_3 = 1$.
The corresponding $6$-dimensional Cayley polytope $ {\rm conv}(\mathcal{A})$ has
$9$ vertices and normalized volume $7$, and the toric multiview ideal equals
\begin{equation}
\label{eq:4binomials}
 J_A \,\,= \,\, \langle z_1 y_3 - x_1 z_3, z_2 x_3 - x_2 z_3,
     z_1 y_2 -y_1 z_2 ,   x_1 y_2 x_3-y_1 x_2 y_3 \rangle. 
     \end{equation}
We note that the quadrics cut out $V_A$ plus an extra component
$\PP^1 \times \PP^1 \times \PP^1$:
\begin{equation}
\label{eq:HAThm56}
\langle z_1 y_3 - x_1 z_3, z_2 x_3 - x_2 z_3,
     z_1 y_2 -y_1 z_2  \rangle \,\, = \,\, J_A     \cap \langle z_1,z_2,z_3 \rangle
\end{equation}
 This equation is precisely \cite[Theorem 5.6]{HA}
 but written in toric coordinates.
     
     \smallskip
     
The toric ideal $J_A$ has precisely $20$ initial monomial ideals, in three symmetry classes,
one for each mixed subdivision of the $3$-dimensional polytope
$$  P \,\, = \,\, {\rm conv}(A_1^T) + {\rm conv}(A_2^T) + {\rm conv}(A_3^T) . $$
Thus $P$ is the Minkowski sum of three of the four triangular facets of the
regular tetrahedron. Each mixed subdivision of $P$
uses one cube $\PP^1 \times \PP^1 \times \PP^1$
and six triangular prisms $\PP^2 \times \PP^1$.
A picture of one of them is seen in Figure~\ref{V3_J8_Blowup}.

\begin{remark} \label{rmk:important}
Our  toric study in this section is universal in the sense that {\bf every}
multiview variety $V_A$ for $n\leq 4$ cameras in linearly general 
position in $\PP^3$ is isomorphic to the toric multiview variety under
a  change of coordinates in $(\PP^2)^n$. This fact can be proved
using the coordinate systems for the Grassmannian ${\rm Gr}(4,3n)$
furnished by the construction in \cite[\S 4]{SZ}.
Here is how it works for $n=4$. The coordinate change via ${\PGL}(3,K)^4$ gives
\begin{equation}
\label{eq:supportset} \,\,
 \bmat{ A_1^T \! & \! A_2^T \! & \! A_3^T \! & \! A_4^T } \,\,=\,\, 
\left[ \,
\begin{matrix} 0 & 0 & 0  \\ * & * & *  \\ * & * & *  \\ * & * & *  \\ \end{matrix} \quad\,\,
\begin{matrix} * & * & *  \\  0 & 0 & 0  \\ * & * & *  \\ * & * & *  \\ \end{matrix} \quad\,\,
\begin{matrix} * & * & *  \\ * & * & *  \\ 0 & 0 & 0  \\ * & * & *  \\ \end{matrix} \quad\,\,
\begin{matrix} * & * & *  \\ * & * & *  \\ * & * & *  \\ 0 & 0 & 0  \\ \end{matrix}
\, \right] \!,
\end{equation}
where the $3 \times 3$-matrices indicated by the stars in the four blocks are invertible.
Now, the $4 \times 12$-matrix (\ref{eq:supportset}) gives a {\em support set} $\Sigma$
that satisfies the conditions in \cite[Proposition 3.1]{SZ}. The corresponding
Zariski open set $\mathcal{U}_\Sigma$ of the Grassmannian ${\rm Gr}(4,12)$
is non-empty. In fact, by \cite[Remark 4.9(a)]{SZ}, the set $\mathcal{U}_\Sigma$
represents configurations whose cameras  $f_1, f_2, f_3,f_4$ are not coplanar.
Now, Theorem 4.6 in \cite{SZ} completes our proof because (the universal Gr\"obner basis of)
the ideal $J_A$ depends only on 
the point in $\mathcal{U}_\Sigma \subset {\rm Gr}(4,12)$ represented by (\ref{eq:supportset})
and not on the specific camera matrices $A_1,\ldots,A_4$. \qed
\end{remark}

\section{Degeneration of Collinear Cameras}

In this section we consider a family of collinear camera positions.
The degeneration of the associated multiview variety will play a key role in
 proving our main results in Section 6, but they may be of independent  interest.
Collinear cameras have been studied in computer vision, for example in \cite{HartleyZisserman}.

Let $\varepsilon $ be a parameter and fix the configuration
$A(\varepsilon) := (A_1, \ldots, A_n)$ where
$$A_i \,:=\ \left[ \begin{array}{cccc} 1 & 1 & 0 & 0 \\ 1 & 0 & 1 & 0 \\ \varepsilon^{n-i} & 0 & 0 & 1 \end{array} \right]$$
 The focal point of camera $i$ is $f_i = (-1:1:1:\varepsilon^{n-i})$ and hence the $n$ cameras given by $A(\varepsilon)$ are collinear in $\PP^3$. Note that these camera matrices stand in sharp contrast to those for which $A$ is generic which was the focus of Sections 2 and 3. They also differ from the toric situation in Section 4. 
 
 We consider the multiview ideal $J_{A(\varepsilon)}$ in the polynomial ring
 $K(\varepsilon)[x,y,z]$, where $K(\varepsilon)$ is the field of rational functions in $\varepsilon$ with coefficients in $K$. 
 Then  $J_{A(\varepsilon)}$ has the Hilbert function (\ref{Vn_Hilbert_function}),
 by Theorem~\ref{thm:multiview ideals on Hilbert scheme}.
  Let $\mathcal{G}_n$ be the set of polynomials in $K(\varepsilon)[x,y,z]$ consisting of the ${n \choose 2}$ quadratic 
polynomials
\begin{equation}
\label{eq:sec5quadrics}
 x_iy_j - x_jy_i \qquad \hbox{for} \,\,\,\, 1 \leq i < j \leq n 
 \end{equation}
and the $3{n \choose 3}$ cubic polynomials below for all choices of $1 \leq i < j < k \leq n$:
\begin{equation}
\label{eq:sec5cubics}
 \begin{array}{c}
 (\varepsilon^{n-k}-\varepsilon^{n-i}) x_iz_jx_k + 
(\varepsilon^{n-j}-\varepsilon^{n-k}) z_ix_jx_k + 
(\varepsilon^{n-i}-\varepsilon^{n-j}) x_ix_jz_k  \\
(\varepsilon^{n-k}-\varepsilon^{n-i}) y_iz_jy_k + 
(\varepsilon^{n-j}-\varepsilon^{n-k}) z_iy_jy_k + 
(\varepsilon^{n-i}-\varepsilon^{n-j})y_iy_jz_k \\
(\varepsilon^{n-k}-\varepsilon^{n-i}) y_iz_jx_k + 
(\varepsilon^{n-j}-\varepsilon^{n-k}) z_iy_jx_k + 
(\varepsilon^{n-i}-\varepsilon^{n-j}) y_ix_jz_k 
\end{array}
\end{equation}
Let $L_n$ be the ideal generated by (\ref{eq:sec5quadrics})
and the following binomials from the first two terms in (\ref{eq:sec5cubics}):
$$  L_n  \,:=\,  \bigl\langle x_iy_j - x_jy_i \,: \, 1 {\leq} i {<} j {\leq}n \bigr\rangle  +   \left\langle 
\begin{array}{c} 
\! x_iz_jx_k - z_ix_jx_k,\\
\! y_iz_jy_k - z_iy_jy_k,\\
\! y_iz_jx_k - z_iy_jx_k
\end{array} : \, 1 {\leq} i  {<} j {<} k {\leq} n \right\rangle \! .$$
Let $N_n$ be the ideal generated by the leading monomials 
in  (\ref{eq:sec5quadrics}) and  (\ref{eq:sec5cubics}):
$$  N_n \, \,:= \,\,  \bigl\langle x_iy_j \,: \,1 {\leq} i {<} j {\leq} n \bigr\rangle
 \, +\,   \bigl\langle 
x_iz_jx_k, \, y_iz_jy_k, \,y_iz_jx_k \,:  \, 1 {\leq} i {<} j {<} k {\leq} n  \bigr\rangle.
$$
The main result in this section is the following construction
of a two-step flat degeneration $J_{A(\varepsilon)} \rightarrow L_n \rightarrow N_n$. This 
gives an explicit realization of~(\ref{eq:TriBiMono}).
We note that $V_{A(\varepsilon)}$ can be seen as a variant of the
{\em Mustafin varieties} in~\cite{CHSW}.

\begin{theorem} \label{thm:sec5main}
The three ideals $J_{A(\varepsilon)}$,  $L_n$ and $N_n$ satisfy the following:
\begin{enumerate}
\item[(a)] The multiview ideal  $J_{A(\varepsilon)}$ is generated by the set $\,\mathcal{G}_n$.
\item[(b)] The binomial ideal $L_n$ equals the special fiber of $J_{A(\varepsilon)}$ for $\varepsilon = 0$.
\item[(c)] The monomial ideal $N_n$ is the initial ideal of $L_n$,
in the Gr\"obner basis sense, with respect to the
lexicographic term order with $ x \succ y \succ z$.
\end{enumerate}
\end{theorem}

 The rest of this section is devoted to explaining and proving these results.
 Let us begin by showing that  $\mathcal{G}_n$ is a subset of $J_{A(\varepsilon)}$.
The determinant of 
$$ A(\varepsilon)_{\{ij\}} \,\, = \,\, \left[ \begin{array}{ccc} A_i & p_i & \bf{0} \\ A_j & \bf{0} & p_j \end{array} \right]$$
equals
$(\varepsilon^{n-j}-\varepsilon^{n-i})( x_iy_j - x_jy_i)$. Hence
$J_{A(\varepsilon)}$ contains (\ref{eq:sec5quadrics}), by the argument in
Lemma~\ref{lem:easyinclusion}.
Similarly, for any $1 {\leq} i {<} j {<} k {\leq} n$, consider the $9 \times 7$ matrix 
$$A(n)_{\{ijk\}} \,\,=\,\, \left[ \begin{array}{ccccccc}
1 & 1 & 0 &  0 & x_i & 0 & 0 \\
1 & 0 & 1 &  0 & y_i & 0 & 0 \\
\varepsilon^{n-i} & 0 & 0 &  1 & z_i & 0 & 0 \\
1 & 1 & 0 &  0 & 0 & x_j  & 0 \\
1 & 0 & 1 &  0 & 0 & y_j & 0 \\
\varepsilon^{n-j} & 0 & 0 &  1 & 0 & z_j  & 0 \\
1 & 1 & 0 &  0 & 0 & 0 & x_k  \\
1 & 0 & 1 &  0 & 0 & 0 & y_k  \\
\varepsilon^{n-k} & 0 & 0 &  1 & 0 & 0 & z_k 
\end{array}
\right].
 $$
The three cubics  (\ref{eq:sec5cubics}), in this order and up to sign, are the determinants of the $7 \times 7$ submatrices of $A(\varepsilon)_{\{ijk\}}$ obtained by deleting the rows corresponding to $y_j$ and $y_k$, the rows corresponding to 
$x_j$ and $x_k$, and the rows corresponding to $x_i$ and $y_k$ respectively. 
We conclude that $\mathcal{G}_n$ lies in $J_{A(\varepsilon)}$.

\smallskip

We next discuss part (b) of Theorem \ref{thm:sec5main}.
Every rational function $c(\varepsilon) \in K(\varepsilon)$ has a unique expansion as a Laurent series 
$c_1\varepsilon^{a_1} + c_2\varepsilon^{a_2} + \cdots $ where $c_i \in K$ and $a_1 < a_2 < \cdots$ are integers. The function ${\rm val}: K(\varepsilon) \rightarrow \ZZ$ 
given by $c(\varepsilon) \mapsto a_1$ is then a valuation on $K(\varepsilon)$, and $K[\![\varepsilon]\!] = 
\{ c \in K(\varepsilon) \,:\, {\rm val}(c) \geq 0 \}$ is its valuation ring. The unique maximal ideal in 
$K[\![\varepsilon]\!]$ is $m = \langle c \in K(\varepsilon) \,:\, {\rm val}(c) > 0 \rangle$.
The residue field $K[\![\varepsilon]\!]/m$ is isomorphic to $K$, so
there is a natural map $K[\![\varepsilon]\!] \rightarrow K$ that represents 
the evaluation at $\varepsilon= 0$.
The {\em special fiber} of an ideal $I \subset K(\varepsilon)[x,y,z]$
is the image of $I \cap K[\![\varepsilon]\!][x,y,z]$ under the induced map
$K[\![\varepsilon]\!][[x,y,z] \rightarrow K[x,y,z]$. The special fiber is denoted ${\rm in}(I)$.
It can be computed from $I$ by a variant of  Gr\"obner bases (cf.~\cite[\S 2.4]{TropicalBook}).

What we are claiming in Theorem \ref{thm:sec5main} (b) is the following identify
$$ {\rm in}( J_{A(\varepsilon)} ) \,\, = \,\, L_n \qquad {\rm in} \,\, K[x,y,z]. $$
It is easy to see that the left hand side contains the right hand side:
indeed, by multiplying the trinomials in (\ref{eq:sec5cubics}) by $\varepsilon^{k-n}$
and then evaluating at $\varepsilon = 0$, we obtain the binomial 
cubics among the generators of $L_n$.

Finally, what is claimed in Theorem \ref{thm:sec5main} (c) is the following identity
$$ {\rm in}_\prec(L_n ) \,\, = \,\, N_n \qquad {\rm in} \,\, K[x,y,z]. $$
Here, ${\rm in}_\prec(L_n)$ is the lexicographic initial ideal
of $L_n$, in the usual Gr\"obner basis sense.
Again, the left hand side contains the right hand side 
because the initial monomials of the binomial generators of $L_n$ generate $N_n$.

Note that $N_n$ is distinct from the generic initial ideal $M_n$.
Even though $M_n$ played a prominent role in  Sections 2 and 3,
the ideal $N_n$ will be more useful in Section 6.
The reason is that $M_n$ is the most singular point on 
the Hilbert scheme $\mathcal{H}_n$ while,
as we shall see, $N_n$ is a smooth point on $\mathcal{H}_n$.

In summary, what we have shown thus far is the following inclusion:
\begin{equation} \label{eq:thusfar}
 N_n \,\, \subseteq \,\,{\rm in}_\prec \bigl( {\rm in}(J_{A(\varepsilon)}) \bigr)  
 \end{equation}
We seek to show that equality holds.
Our proof rests on the following lemma.

\begin{lemma}
\label{lem:N_Hilbert_function}
The monomial ideal $N_n$ has the $\ZZ^n$-graded Hilbert function~\eqref{Vn_Hilbert_function}.
\end{lemma}

\begin{proof} 
Let $u=(u_1,\ldots,u_n)\in\NN^n$, and let $\mathfrak{B}_u$ be the set of all monomials
of multidegree $u$ in $K[x,y,z]$ which are not in $N_n$.  We need to show that
$$|\mathfrak{B}_u| \,=\, {u_1+\cdots+u_n+3\choose 3} - \sum_{i=1}^n{u_i+2\choose 3}.$$

It can be seen from the generators of $N_n$ that the monomials in $\mathfrak{B}_u$ are
of the form $z^a y^bx^c z^d$ for $a,b,c,d\in\NN^n$ such that $u=a+b+c+d$ and 
\begin{align*}
a &= (a_1,\ldots, a_i, 0,\ldots,0)\\
b &= (0,\ldots,0,b_i,\ldots,b_j,0,\ldots,0)\\
c &= (0,\ldots,0,c_j,\ldots, c_k, 0,\ldots,0)\\
d &= (0,\ldots, 0,d_k, \ldots, d_n)
\end{align*}
for some triple $i,j,k$ with $1\le i\le j\le k\le n$.

We count the monomials in $\mathfrak{B}_u$ using a combinatorial ``stars and bars'' argument.
Each monomial can be formed in the following way.
Suppose there are $u_1+\cdots+u_n+3$ blank spaces laid left to right.
Fill exactly three spaces with bars. This leaves $u_1+\cdots+u_n$
open blanks to fill in, which is the total degree of a monomial in $\mathfrak{B}_u$.
The three bars separate the blanks into four
compartments, some possibly empty. From these compartments we greedily form $a$, $b$, $c$, and $d$ to make
$z^ay^bx^cz^d$ as described below.

In what follows, $\star$ is
used as a placeholder symbol.  Fill the first $u_1$ blanks with the symbol $\star_1$, the next
$u_2$ blanks with $\star_2$, and continue to fill up until the last $u_n$ blanks are filled with $\star_n$. Now we pass
once more through these symbols and replace each $\star_i$ with either $x_i$, $y_i$, or $z_i$ such that 
all variables in the first compartment are $z$'s, those in the second are $y$'s, then $x$'s and in the fourth compartment  $z$'s. Removing the bars gives  $z^ay^bx^cz^d$ in $\mathfrak{B}_u$.

There are $\displaystyle{u_1+\cdots+u_n+3\choose 3}$ ways of choosing the three bars.
The monomials in
$\mathfrak{B}_u$  are overcounted only when $i=j=k$ if $z_i$ appears in both the 
first and fourth compartments.
Indeed, in such cases if we require $a_i=0$, the monomial is uniquely represented, so we
are overcounting by the $\displaystyle{u_i+2\choose 3}$ choices when $a_i\ne 0$.
\end{proof}

We are now prepared to derive the main result of this section.  

\medskip

\noindent  {\em Proof of Theorem \ref{thm:sec5main}:}
Lemma~\ref{lem:N_Hilbert_function} and Theorem~\ref{thm:multiview ideals on Hilbert scheme} tell us that
$N_n$ and $J_{A(\varepsilon)}$ have the same $\ZZ^n$-graded Hilbert function \eqref{Vn_Hilbert_function}.
We also know from \cite[\S 2.4]{TropicalBook} that $\tin(J_{A(\varepsilon)})$ has the same Hilbert function, just as 
passing to an initial monomial ideal for a term order preserves Hilbert function.
Hence the equality
$N_n \,\, \subseteq \,\,{\rm in}_\prec \bigl( {\rm in}(J_{A(\varepsilon)}) \bigr)  $ holds in (\ref{eq:thusfar}).
This proves parts (b) and (c).
We have shown that $\mathcal{G}_n$ is a Gr\"obner basis
for the homogeneous ideal $J_{A(\varepsilon)}$ in the
valuative sense of  \cite[\S 2.4]{TropicalBook}. This implies
that $\mathcal{G}_n$ generates  $J_{A(\varepsilon)}$.
\qed
\medskip

\begin{remark} \label{rmk:decomp}
The polyhedral subcomplexes of $(\Delta_2)^n$ defined by
the binomial ideal $L_n$ and the monomial ideal $N_n$ are
combinatorially interesting. For instance,
$L_n$ has prime decomposition $I_3\cap I_4\cap\cdots\cap I_n\cap I_{n+1}$, where
\begin{align*}
I_t \,\,&:= \,\langle\, x_i,y_i :\, i=t,t+1,\ldots,n \,\rangle \ +\\
&\ \ \ \ \ \langle\, x_iy_j - x_jy_i :  1\le i < j < t \,\rangle\  + \\
&\ \ \ \ \ \langle\, x_iz_j - x_jz_i, \,y_iz_j - y_jz_i : 1\le i < j < t-1 \,\rangle.
\end{align*}
The monomial ideal $N_n$ is the intersection of
 ${\rm in}_\prec(I_t)$ for $t=3,\ldots,n+1$. \qed
\end{remark}

\section{The Hilbert Scheme}

We define $\mathcal{H}_n$ to be the multigraded Hilbert scheme which parametrizes all $\ZZ^n$-homogeneous ideals in $K[x,y,z]$ with the Hilbert function in (\ref{Vn_Hilbert_function}).
According to the general construction given in \cite{HaimanSturmfels}, 
 $\mathcal{H}_n$ is a projective scheme.
The ideals $J_A$ and ${\rm in}_\prec(J_A)$ for $n$ distinct
camera positions, as well as the combinatorial ideals $M_n, L_n$ and $N_n$ all correspond to closed points  on $\mathcal{H}_n$.

Our Hilbert scheme $\mathcal{H}_n$ is closely related to the  Hilbert scheme $H_{4,n}$ 
which was studied in \cite{CS}. We already utilized results 
 from that paper in our proof of Theorem \ref{thm:UGB}.
 Note that $H_{4,n}$ parametrizes degenerations of the
 diagonal $\PP^3$ in $(\PP^3)^n$ while
 $\mathcal{H}_n$ parametrizes blown-up images of that $\PP^3$
 in $(\PP^2)^n$.
 
 Let $G=\PGL(3,K)$ and $\B\subset G$ the Borel subgroup of lower-triangular $3\times 3$ matrices modulo scaling. 
 The group $G^n$ acts on $K[x,y,z]$ and this induces an action on the Hilbert scheme $\mathcal{H}_n$.
 Our results concerning the ideal $M_n$ in Section 3 imply the following corollary,
 which summarizes the statements analogous to Theorem 2.1 and Corollaries 2.4 and 2.6 in \cite{CS}.

\begin{corollary}
The multigraded Hilbert scheme $\mathcal{H}_n$ is connected.
The point representing the generic initial ideal $M_n$ lies on each irreducible component of $\mathcal{H}_n$.
All ideals that lie on $\mathcal{H}_n$ are radical and Cohen-Macaulay.
\end{corollary}

In particular, every monomial ideal in $\mathcal{H}_n$ is squarefree
and can hence be identified with its variety in $(\PP^2)^n$,
or, equivalently, with a subcomplex in  the product of triangles $(\Delta_2)^n$.
One of the first questions one asks about any multigraded Hilbert scheme,
including $\mathcal{H}_n$, is to list its monomial ideals.

This task is easy for the first case, $n=2$. The Hilbert scheme $\mathcal{H}_2$ 
parametrizes $\ZZ^2$-homogeneous ideals in $K[x,y,z]$ having Hilbert function 
$$ h_2:\NN^2\to\NN, \, (u_1,u_2)\mapsto {u_1+u_2+3\choose 3} - {u_1+2\choose 3} - {u_2+2\choose 3}. $$
There are exactly nine monomial ideals on $\mathcal{H}_2$, namely
$$ \langle x_1x_2 \rangle ,\ \langle x_1y_2 \rangle ,\ \langle x_1z_2 \rangle ,\ \langle y_1x_2 \rangle ,\ 
\langle y_1y_2 \rangle ,\ \langle y_1z_2 \rangle ,\ \langle z_1x_2 \rangle ,\ \langle z_1y_2 \rangle ,\ \langle z_1z_2 \rangle  .$$
In fact, the ideals on $\mathcal{H}_2$ are precisely the
principal ideals generated by bilinear forms, and
$\mathcal{H}_2$ is isomorphic to an $8$-dimensional projective space
$$ \mathcal{H}_2 \,=\, \{ \langle c_0x_1x_2+c_1x_1y_2+\cdots+c_8z_1z_2 \rangle \
 \,:                                                                            
\, (c_0:c_1:\cdots:c_8)\in\PP^8\}.$$

The principal ideals $J_A$ which actually arise from two cameras
form a cubic hypersurface in this $\mathcal{H}_2 \simeq \PP^8$.  To see this, we
write $A^j_i$ for the $j$-th row of the $i$-th camera matrix
and $[A_{i_1}^{j_1} A_{i_2}^{j_2} A_{i_3}^{j_3}  A_{i_4}^{j_4}]$
for the $4 \times 4$-determinant formed by four such row vectors.
The bilinear form can be written as
$$ \xx_2^T F \xx_1 \,= \,
\bmat{x_2 & y_2 & z_2}
\bmat{c_0 & c_3 & c_6\\ c_1 & c_4 & c_7 \\ c_2 & c_5 & c_8}
\bmat{x_1\\y_1\\z_1},
$$
where $F$ is the \emph{fundamental matrix} \cite{HartleyZisserman}. In terms of the camera matrices, 
\begin{equation}
\label{eq:fundmatrix}
 F \, = \,
\bmat{
 \phantom{-}[A_1^2 A_1^3 A_2^2 A_2^3 ] & - [A_1^1 A_1^3 A_2^2 A_2^3 ] &  \phantom{-}[A_1^1 A_1^2 A_2^2 A_2^3 ] \\
 -[A_1^2 A_1^3 A_2^1 A_2^3 ] &  \phantom{-}[A_1^1 A_1^3 A_2^1 A_2^3 ] & -[A_1^1 A_1^2 A_2^1 A_2^3 ] \\
  \phantom{-}[A_1^2 A_1^3 A_2^1 A_2^2 ] & -[A_1^1 A_1^3 A_2^1 A_2^2 ] &  \phantom{-}[A_1^1 A_1^2 A_2^1 A_2^2 ] }.
 \end{equation}
 This matrix has rank $\leq 2$, and every $3 \times 3$-matrix of rank $\leq 2$ can be
 written in this form for suitable camera matrices $A_1$ and $A_2$ of size $3 \times 4$.

The formula in (\ref{eq:fundmatrix}) defines a map $(A_1,A_2) \mapsto F$
from pairs of camera matrices with distinct focal points into the Hilbert scheme $\mathcal{H}_2$.
The closure of its image is a compactification of the space of camera positions.
We now precisely define the corresponding map for arbitrary $n \geq 2$.
The construction is inspired by the construction due to Thaddeus discussed in
\cite[Example 7]{CS}.

\smallskip

Let ${\rm Gr}(4,3n)$ denote the Grassmannian of
$4$-dimensional linear subspaces of $K^{3n}$.
The $n$-dimensional algebraic torus $(K^*)^n$ acts on
this Grassmannian by scaling the coordinates on $K^{3n}$,
where the $i$th factor $K^*$ scales the coordinates
indexed by $3i-2, 3i-1$ and $3i$. Thus, if we represent
each point in ${\rm Gr}(4,3n)$ as the row space of a
$(4 \times 3n)$-matrix $ \bmat{ A_1^T \! & \! A_2^T \! & \!  \cdots \! & \! A_n^T }$, then
$\lambda = (\lambda_1,\ldots,\lambda_n) \in (K^*)^n$ sends this matrix to
 $ \bmat{ \lambda_1 A_1^T \! & \! \lambda_2 A_2^T \! & \!  \cdots \! & \! \lambda_n A_n^T }$.
The multiview ideal $J_A$ is invariant under this action by $(K^*)^n$. In symbols,
$J_{\lambda \circ A} = J_A$. In the next lemma,
 GIT stands for {\em geometric invariant theory}.

\begin{lemma} \label{lem:cameramap}
The assignment $\, A \mapsto J_A \,$ defines
 an injective rational map $\gamma$ from a
GIT quotient $ {\rm Gr}(4,3n) /\!/  (K^*)^n$ to the
multigraded Hilbert scheme~$\mathcal{H}_n$.
\end{lemma}

\begin{proof}
For the proof it suffices to check that $J_A \not= J_{A'}$ whenever
$A$ and $A'$ are generic camera configurations
that are not in the same $(K^*)^n$-orbit.
\end{proof}

We call $\gamma$ the camera map. Since we need
$\gamma$ only as a rational map, the choice of
linearization does not matter when we form the GIT quotient.
The closure of its image in $\mathcal{H}_n$ is well-defined
and independent of that choice of linearization.
We define the {\em compactified camera space}, for $n$ cameras, to be
$$ \Gamma_n \,\,:=\,\, \overline{\gamma({\rm Gr}(4,3n) /\!/ (K^*)^n)} \,\,\,\subseteq \,\, \mathcal{H}_n. $$
The projective
variety $\Gamma_n$ is a natural compactification of
the parameter space studied by Heyden in \cite{Heyden}. 
Since the torus $(K^*)^n$ acts on ${\rm Gr}(4,3n)$ with a one-dimensional
stabilizer, Lemma \ref{lem:cameramap} implies that the
compactified space of $n$ cameras has the dimension we expect from \cite{Heyden}, namely,
$$  {\rm dim}(\Gamma_n) \,\, = \,\,
{\rm dim}({\rm Gr}(4,3n)) - (n - 1) \,\, = \,\,  4(3n-4) - (n - 1) \,\, = \,\, 11 n - 15. $$

We regard the following theorem as the main result in this paper.

\begin{theorem}
\label{thm:component}
For $n \geq 3$, the compactified  camera space  $\Gamma_n$ appears as
a distinguished irreducible component in the  multigraded Hilbert scheme $\mathcal{H}_n$.
\end{theorem}

Note that the same statement if false for $n=2$:
$\Gamma_2$ is not a component of $\mathcal{H}_3 \simeq \PP^8$.
It is the hypersurface consisting of the fundamental matrices~(\ref{eq:fundmatrix}).

\begin{proof} By definition,  the compactified camera space $\Gamma_n$ is a closed subscheme of $\mathcal{H}_n$.
The discussion above shows that the dimension of any
irreducible component of $\mathcal{H}_n$  that contains
 $\Gamma_n$ is no smaller than $11n-15$.
We shall now prove the same $11n-15$ as an
upper bound for the dimension. This is done by exhibiting
a point in $\Gamma_n$ whose tangent space in 
the Hilbert scheme $\mathcal{H}_n$
has dimension $11n-15$. This will imply the assertion.

For any ideal $I\in\mathcal{H}_n$, the tangent space to
the Hilbert scheme $\mathcal{H}_n$ at $I$ is the space
of $K[x,y,z]$-module homomorphisms $I\to K[x,y,z]/I$ of degree~{\bf 0}.
In symbols, this space is $\,{\rm Hom}(I,K[x,y,z]/I)_{\bf 0} $.
The $K$-dimension of the tangent space provides an upper bound
for the dimension of any component on which $I$ lies.
It remains to specifically identify a point on $\Gamma_n$ that is smooth on 
$\mathcal{H}_n$, an ideal which has tangent space dimension exactly $11n-15$.

It turns out that the monomial ideal $N_n$ described in 
the previous section has this desired property.
Lemmas \ref{lem:on_component} and \ref{lem:tangent_space}
below give the details.
\end{proof}

\begin{lemma}
\label{lem:on_component}
The ideals $L_n$ and $N_n$ from the previous section lie in  $\Gamma_n$.
\end{lemma}

\begin{proof}
The image of $\gamma$ in $\mathcal{H}_n$ consists of
all multiview ideals $J_A$, where $A$ runs over configurations of $n$ 
distinct cameras, by Theorem \ref{thm:multiview ideals on Hilbert scheme}.
Let $A(\varepsilon)$ denote the collinear configuration in Section 5,
and consider any specialization of $\varepsilon$
to a non-zero scalar in $K$. The resulting ideal
$J_{A(\varepsilon)}$ is a  $K$-valued point of $\Gamma_n$, for any 
$\varepsilon \in K \backslash \{0\}$. The special fiber
$J_{A(0)}  = L_n$ is in the Zariski closure of these points, 
because, locally, any regular function vanishing on
the coordinates of $J_{A(\varepsilon)}$ for 
all $\varepsilon \not= 0$ will vanish for $\varepsilon = 0$.
We conclude that $L_n$ is a $K$-valued point in the projective variety $\Gamma_n$.
Likewise, since $N_n = {\rm in}_\prec(L_n)$ is
an initial monomial ideal of $L_n$, it also lies on $\Gamma_n$.
\end{proof}

\begin{lemma}
\label{lem:tangent_space}
The tangent space of the multigraded Hilbert scheme $\mathcal{H}_n$ at the
point represented by the
monomial ideal $N_n$ has dimension $11n-15$.
\end{lemma}

\begin{proof}
The tangent space at $N_n$ equals 
$\,{\rm Hom}(N_n,K[x,y,z]/N_n)_{\bf 0} $.
We shall present a basis for this space that is broken into three
distinct classes: those homomorphisms that act nontrivially only on the quadratic generators,
those that act nontrivially only on the cubics, and those with a mix of both.

Each $K[x,y,z]$-module homomorphism $\varphi:N_n\to K[x,y,z]/N_n$ below is described by its action
on the minimal generators of $N_n$.  Any generator
not explicitly mentioned is mapped to 0 under $\varphi$.
One checks that each is in fact a well-defined $K[x,y,z]$-module homomorphism 
from $N_n$ to $K[x,y,z]/N_n$.

\smallskip

\underline{Class I:} \
For each $1\le i < n$, we define the following maps
\begin{itemize}
\item $\alpha_i: x_iy_k\mapsto y_iy_k$ for all $i<k\le n$,
\item $\beta_i: x_iy_{i+1}\mapsto x_{i+1}y_i$.
\end{itemize}
For each $1<k\le n$, we define the following map
\begin{itemize}
\item $\gamma_k : x_iy_k \mapsto x_ix_k$ for all $1\le i < k$.
\end{itemize}
We define two specific homomorphisms
\begin{itemize}
\item $\delta_1: x_1y_2 \mapsto y_1z_2$,
\item $\delta_2:x_{n-1}y_n\mapsto z_{n-1}x_n$.
\end{itemize}

\smallskip

\underline{Class II:} \
For each $1<j<n$, we define the following maps.  Each
homomorphism is defined on every pair $(i,k)$ such that
$1\le i < j < k\le n$.
\begin{itemize}
\item $\rho_j: x_iz_jx_k\mapsto x_ix_jx_k$ and $y_iz_jx_k \mapsto y_ix_jx_k$,
\item $\sigma_j: x_iz_jx_k \mapsto x_ix_jz_k$ and $y_iz_jx_k\mapsto y_ix_jz_k$,
\item $\tau_j: x_iz_jx_k \mapsto x_iz_jz_k$ and $y_iz_jx_k \mapsto y_iz_jz_k$,
\item $\nu_j: y_iz_jx_k \mapsto y_iy_jx_k$ and $y_iz_jy_k \mapsto y_iy_jy_k$,
\item $\mu_j: y_iz_jx_k \mapsto z_iy_jx_k$ and $y_iz_jy_k \mapsto z_iy_jy_k$,
\item $\pi_j: y_iz_jx_k \mapsto z_iz_jx_k$ and $y_iz_jy_k \mapsto z_iz_jy_k$.
\end{itemize}

\underline{Class III:} \
For each $1\le i < n$, we define the  map
\begin{itemize}
\item $\epsilon_i: x_iy_k \mapsto z_iy_k$ and $x_iz_jx_k \mapsto z_iz_jx_k$ \
 for $i<k\le n$ and $i<j<k$.
\end{itemize}
For each $1<k\le n$, we define the  map
\begin{itemize}
\item $\zeta_k: x_iy_k \mapsto x_iz_k$ and $y_iz_jy_k\mapsto y_iz_jz_k$ \
for $1\le i < k$ and  $i<j<k$.
\end{itemize}

\smallskip
All these maps are linearly independent over the field $K$. There are $n-1$
maps each of type $\alpha_i$, $\beta_i$, $\gamma_k$, $\epsilon_i$, and $\zeta_k$,
for a total of $5(n-1)$ different homomorphisms.  Each subclass of maps in class II
has $n-2$ members, adding $6(n-2)$ more homomorphisms.  Finally 
adding $\delta_1$ and $\delta_2$, we arrive at the total count of
$5(n-1)+6(n-2)+2 = 11n-15$ homomorphisms. 

 We claim that any $K[x,y,z]$-module
 homomorphism $N_n\to K[x,y,z]/N_n$ can be recognized as a $K$-linear combination
 of those from the three classes described above.
To prove this, suppose that $\varphi:N_n\to K[x,y,z]/N_n$ is a module homomorphism. 
For $1\le i < k \le n$, we can write $\varphi(x_iy_k)$ as a linear combination of monomials
of multidegree $e_i+e_k$ which are not in $N_n$.  By subtracting appropriate multiples of
$\alpha_i$, $\epsilon_i$, $\gamma_k,$ and $\zeta_k$, we can assume that
$$\varphi(x_iy_k) = a\,y_ix_k + b\,y_iz_k + c\,z_ix_k + d\,z_iz_k$$
for some scalars $a,b,c,d\in K$. We show that this can be written as a linear combination
of the maps described above by considering a few cases.

In the first case we assume $i+1<k$.
We use $K[x,y,z]$-linearity to infer 
$$\varphi(x_iy_{i+1}y_k) = a\,y_iy_{i+1}x_k + b\,y_iy_{i+1}z_k + c\,z_iy_{i+1}x_k + d\,z_iy_{i+1}z_k = y_k\, \varphi(x_iy_{i+1}).$$
Specifically, $y_k$ divides the middle polynomial.  But none of the four monomials are
zero in the quotient $K[x,y,z]/N_n$.  Hence, $0=a=b=c=d$.

For the subsequent cases we assume $k=i+1$.  This allows us to further
assume that $a=0$, since we can subtract off $a\, \beta_i(x_iy_{i+1})$.
Now suppose that we have strict inequality $k<n$.
As before, the $K[x,y,z]$-linearity of $\varphi$ gives 
$$\varphi(x_iy_ky_n) = d\,z_iz_ky_n = y_k\,\varphi(x_iy_n).$$
Specifically, $y_k$ divides the middle term.  Hence, $d=0$.  Similarly, $c=0$:
$$\varphi(x_iy_kz_kx_n) = c\,z_ix_kz_kx_n = y_k\,\varphi(x_iz_kx_n).$$
Suppose we further have the strict inequality $1<i$.  Then necessarily $b=0$:
$$\varphi(y_1z_ix_iy_k) = b\, y_1z_iy_iz_k = x_i\,\varphi(y_1z_iy_k).$$
However, if $i=1$ and $k=2$, we have that $\varphi(x_1y_2) = b\, \delta_1(x_1y_2)$.

The only case that remains is $k=n$ and $i=n-1$. Here, we can also assume
that $c=0$ by subtracting $c\, \delta_2(x_{n-1}y_n)$.  We will show that $d=0=b$
by once more appealing to the fact that $\varphi$ is a module homomorphism:
$$\varphi(x_1x_{n-1}y_n) = d\, x_1z_{n-1}z_n = x_{n-1}\, \varphi(x_1y_n),$$
which gives $d=0$.  This subsequently implies the desired $b=0$, because
$$\varphi(y_1x_iz_iy_n) = b\, y_1y_iz_iz_n = x_i\, \varphi(y_1z_iy_n).$$

This has finally put us in a position where we can assume that
$\varphi(x_iy_k)=0$ for all $1\le i < k \le n$. To finish the proof that $\varphi$
is  a linear combination of the $11n-15$ classes described above,
we need to examine what happens with the cubics.
Suppose $1\le i<j<k\le n$, and consider $\varphi(y_iz_jx_k)$.
This can be written as a linear sum of the 17 standard monomials
of multidegree $e_i+e_j+e_k$ which are not in $N_n$.
Explicitly, these standard monomials are:
$$\begin{array}{lllll}
x_ix_jx_k, & x_ix_jz_k, & x_iz_jz_k, & y_ix_jx_k,& y_ix_jz_k\\
y_iy_jx_k, & y_iy_jy_k, & y_iy_jz_k, & y_iz_jz_k,\\
z_ix_jx_k, & z_ix_jz_k, & z_iy_jx_k, & z_iy_jy_k,\\
z_iy_jz_k, & z_iz_jx_k, & z_iz_jy_k, & z_iz_jz_k.
\end{array}$$
By subtracting off multiples of the maps $\rho_j$, $\sigma_j$,
$\tau_j$, $\nu_j$, $\mu_j$, and $\pi_j$, we can assume that this is
a sum of the 11 monomials remaining after removing $y_ix_jx_k$,
$y_ix_jz_k$, $y_iz_jz_k$, $y_iy_jx_k$, $z_iy_jx_k$, and $z_iz_jx_k$.
However, now note  that
$$\varphi(x_iy_iz_jx_k) = x_i\,\varphi(y_iz_jx_k) = y_i\,\varphi(x_iz_jx_k).$$
This means that for every one of the 11 monomials $m$ appearing in the sum,
either $x_im=0$ or $y_i$ divides $m$. Similarly,
$$\varphi(y_iz_jx_ky_k) = y_k\,\varphi(y_iz_jx_k) = x_k\,\varphi(y_iz_jy_k),$$
and so either $y_km=0$ or $x_k$ divides $m$.
Taking these both into consideration actually kills every one
of the 11 possible standard monomials (we spare the reader the explicit check),
and hence we can assume that $\varphi(y_iz_jx_k)=0$.

Now consider what happens with $\varphi(x_iz_jx_k)$.  Indeed,
$$0=x_i\,\varphi(y_iz_jx_k) = \varphi(x_iy_iz_jx_k)=y_i\,\varphi(x_iz_jx_k).$$
So for every one of the 17 standard monomials $m$ which possibly appears
in the support of $\varphi(x_iz_jx_k)$ we must have that $y_im=0$
in $K[x,y,z]/N_n$. This actually leaves us with only two possible
such standard monomials -- namely $z_iz_jx_k$ and $z_iz_jy_k$.
We write $\varphi(x_iz_jx_k)=a\, z_iz_jx_k + b\, z_iz_jy_k$.

The fact that we assume $\varphi(x_iy_k)=0$ implies $a=0=b$.
This is because
$$0 = z_jx_k\, \varphi(x_iy_k) = \varphi(x_iz_jx_ky_k) = y_k\,\varphi(x_iz_jx_k).$$
To sum up, we have shown that, under our assumptions,
if $\varphi(y_iz_jx_k)=0$ holds then it also must be the case that
$\varphi(x_iz_jx_k)=0$.  We can prove in a similar manner that
$\varphi(y_iz_jy_k)=0$, and this finishes the proof that $\varphi$
can be written as a $K$-linear sum of the $11n-15$ classes of maps described.
\end{proof}

We reiterate that Theorem \ref{thm:component} fails for $n=2$, since 
$\mathcal{H}_2\simeq \PP^8$, and $\Gamma_2$ is a cubic hypersurface
cutting through $\mathcal{H}_2$. We offer a short report for $n=3$.

\begin{remark} \label{rmk:mysterious}
The Hilbert scheme $\mathcal{H}_3$ contains $13,824$ monomial ideals.
These come in $16$ symmetry classes under the action of $(S_3)^3\rtimes S_3$.
A detailed analysis of these symmetry classes and how we found the
$13,824$ ideals appears on the website
\texttt{www.math.washington.edu/$\sim$aholtc/HilbertScheme}.
For seven of the symmetry classes, the tangent space dimension
is less than ${\rm dim}(\Gamma_3) = 18$. From this
we infer that $\mathcal{H}_3$ has components other than~$\Gamma_3$.

We note that the number $13,824$ is exactly the number of monomial ideals
on $H_{3,3}$ as described in \cite{CS}. Moreover, the monomial ideals on $H_{3,3}$
also fall into $16$ distinct symmetry classes.
We do not yet fully understand the relationship between
$\mathcal{H}_n$ and $H_{3,n}$ suggested by this observation.

Moreover, it would be desirable to coordinatize 
the inclusion $\Gamma_3 \subset \mathcal{H}_3$ and to relate it
to the equations defining {\em trifocal tensors}, as seen in \cite{AT, Heyden}.
It is our intention to investigate this topic in a subsequent publication. 
\end{remark}

Our study was restricted to cameras that take $2$-dimensional
pictures of $3$-dimensional scenes. Yet, residents of {\em flatland} might be more interested
in taking $1$-dimensional pictures of $2$-dimensional scenes.
From a mathematical perspective, generalizing to arbitrary dimensions makes sense:
given  $n$ matrices of format $r \times s$ we get a map
from $\PP^{s-1} $ into $(\PP^{r-1})^n$, and one could study
the Hilbert scheme parametrizing the resulting varieties.
Our focus on $r=3$ and $s=4$
was motivated by the context of computer vision.

\end{document}